\documentclass[draft,11pt, a4paper, sort&compress]{elsarticle}

\usepackage{amsmath}
\usepackage{amssymb}
\usepackage{amsthm}
\usepackage{mathrsfs}
\usepackage{amsfonts}
\usepackage{enumerate}
 \usepackage[protrusion=true,expansion=true]{microtype}
 \usepackage[numbers]{natbib}
    \usepackage[varg]{txfonts}

\usepackage[hmarginratio=1:1,top=32mm,columnsep=20pt]{geometry}

\newtheorem{theorem}{Theorem}[section]
 \newtheorem{corollary}[theorem]{Corollary}
 \newtheorem{lemma}[theorem]{Lemma}
 \newtheorem{proposition}[theorem]{Proposition}
 \newtheorem{definition}[theorem]{Definition}
 \newtheorem{remark}[theorem]{Remark}

 \newcommand{\eps}{\varepsilon}
 \newcommand{\norm}[1]{\Vert#1\Vert}
 \newcommand{\normm}[1]{\Vert|#1\Vert|}
 \newcommand{\abs}[1]{\left\vert#1\right\vert}
 
 \newcommand{\set}[1]{\left\{#1\right\}}
  
 \newcommand{\inner}[1]{\left(#1\right)}
 
 \newcommand{\comi}[1]{ \left<#1 \right>}

\makeatletter
\def\ps@pprintTitle{%
     \let\@oddhead\@empty
     \let\@evenhead\@empty
     \def\@oddfoot{\reset@font\hfil\thepage\hfil}
     \let\@evenfoot\@oddfoot}
 \makeatother

\begin{document}

\begin{frontmatter}

\title
 {Global Well-posedness of a Prandtl Model from MHD in Gevrey Function Spaces}

%----------Author 1
  \author[ad1,ad2]{Wei-Xi Li}
\ead{wei-xi.li@whu.edu.cn}

  \author[ad1]{Rui Xu}
\ead{xurui218@whu.edu.cn}

\author[ad3]{Tong Yang}
  \ead{matyang@cityu.edu.hk}

\address[ad1]{School of Mathematics and Statistics, Wuhan University, Wuhan
430072, China}
\address[ad2]{Hubei Key Laboratory of Computational Science, Wuhan University, 430072 Wuhan, China}

\address[ad3]{
	Department of Mathematics, City University of Hong Kong, Hong Kong
  \bigskip\\
{\it \large Dedicated to Professor Banghe Li on the occasion of his 80th birthday}}

 \vspace{.5in}
  
\begin{abstract}
We consider  a Prandtl model derived from MHD in the Prandtl-Hartmann regime that
has a damping term due to the effect of the Hartmann boundary layer. A global-in-time well-posedness is obtained in the Gevrey function space with the optimal index  $2$.   The proof is based on a cancellation mechanism through some auxiliary functions from the study of the Prandtl equation and an observation about the structure of the loss of one order tangential derivatives through twice operations of the  Prandtl operator.  
\end{abstract}

\begin{keyword}
Magnetic Prandtl equation, Gevrey function space, global well-posedness, auxiliary functions, 
loss of derivative
\MSC[2020] 76W05,35M33, 35Q35 
\end{keyword}

\end{frontmatter}

 \tableofcontents
 
\section{Introduction and main result}

We study a Prandtl type system  with a damping term induced by a Hartmann magnetic field
that was derived by G\'erard-Varet and Prestipino  in
\cite{MR3657241}.  Suppose that the fluid domain is $\mathbb R_+^n=\big\{(x,y)\in\mathbb R^n;\ x\in\mathbb R^{n-1}, y>0\big\} $. Denote by $ \boldsymbol{u}$ and $v$ the tangential and normal components of the velocity fields.    Then the dimensionless   magnetic Prandtl model in $\mathbb R_+^n $    is
given by
\begin{equation}\label{dpe+}
	\left\{
\begin{aligned}
&(\partial_t   +  \boldsymbol{u}\cdot\partial_x +v\partial_{y} -\partial_{y}^2) \boldsymbol{u} +\partial_xp+ \boldsymbol{u}-\boldsymbol{U} =0,\\
&\partial_x\cdot  \boldsymbol{u}+\partial_{y}v=0,\\
&(  \boldsymbol{u}, v)|_{y=0}=(  \boldsymbol{0}  , 0),  \quad \lim_{y\rightarrow+\infty}  \boldsymbol{u}= \boldsymbol{U},\\
& \boldsymbol{u}|_{t=0}= \boldsymbol{u}_0,
\end{aligned}
\right.
\end{equation}
where $\partial_xp$ and $ \boldsymbol{U}$ are the traces of a given Euler flow satisfying  the Bernoulli law
\begin{eqnarray*}
	\inner{\partial_t  + \boldsymbol{U}\cdot \partial_x }\boldsymbol{U}+\partial_xp=0.
\end{eqnarray*}
System \eqref{dpe+} is derived from   the  incompressible MHD system in the mixed Prandtl/Hartmann regime, where the  leading order  equations are  \begin{equation}\label{leading}
		\left\{
\begin{aligned}
&(\partial_t   +  \boldsymbol{u}\cdot\partial_x +v\partial_{y} -\partial_{y}^2) \boldsymbol{u} +\partial_xp=\partial_y \boldsymbol{b},\\
&\partial_y\boldsymbol{u}+\partial_y^2\boldsymbol{b}=0,\\
&\partial_x\cdot  \boldsymbol{u}+\partial_{y}v=0,\\
&(  \boldsymbol{u}, v)|_{y=0}=(  \boldsymbol{0}  , 0),  \quad \lim_{y\rightarrow+\infty}  (\boldsymbol{u}, \boldsymbol{b}) = (\boldsymbol{U}, \boldsymbol{B}), \\
& \boldsymbol{u}|_{t=0}= \boldsymbol{u}_0,
\end{aligned}
\right.
\end{equation}
where $\boldsymbol{b}$ stands for  the tangential magnetic field satisfying that $\partial_y \boldsymbol{b}$ trends to $0$ as $y\rightarrow+\infty.$ By   the second   equation  in \eqref{leading}  and the boundary conditions of $(\boldsymbol{u}, \boldsymbol{b})$,  the magnetic fields $ \boldsymbol{b}$ can be determined in terms of $\boldsymbol{u}$; that is, 
\begin{eqnarray*}
	\partial_y\boldsymbol{b}=-(\boldsymbol{u}-
	\boldsymbol{U}), \quad  \boldsymbol{b}=\boldsymbol{B}+\int_y^{+\infty}(\boldsymbol{u}-\boldsymbol{U}) dy.   
\end{eqnarray*}
Then system \eqref{leading} reduces to \eqref{dpe+}. 
We refer to  %D. G\'erard-Varet and M. Prestipino's work 
\cite{MR3657241} for the detailed argument.   Compared with the classical Prandtl system,  there is an extra  damping term in \eqref{dpe+};  this does not lead to any additional difficulty in the local-in-time existence and uniqueness theory.  Hence,  the local  well-posedness theories of  the classical Prandtl  system  in Sobolev or Gevrey function spaces established in  \cite{MR3327535,MR3925144,MR3385340,lmy}     hold  for \eqref{dpe+}.   
On the other hand,   we can expect a global-in-time    solution to \eqref{dpe+}
because of the damping effect. In fact, some results in this direction were obtained
in    Xie-Yang \cite{MR3873032} and  Chen-Ren-Wang-Zhang \cite{MR4169212}  in the analytic and Sobolev spaces, respectively, about the 2D global stability of the Hartmann layer that satisfies 
  Oleinik's monotonicity condition.   
   This paper aims to study the  global-in-time property of the system  for general data without any structural assumption. 

Compared with the local-in-time well-posedness theory  (see, for instance,  \cite{MR3327535, MR3795028,   MR3600083, MR3925144,  MR1476316, MR2601044, MR3429469, MR2849481,MR3461362, MR3284569,MR3493958, MR4055987,MR2020656, MR3710703, 2022arXiv220110139Y, MR3464051}   and references therein), the global-in-time property of the Prandtl system is much less   known so that  it is far from  being well understood.  Here we refer to   an early work on    weak solutions by Xin-Zhang \cite{MR2020656},
% and the very recent progress  \cite{2022arXiv220308988X} on the regularity  of weak global solutions,  
and a work
on analytic solutions by Paicu-Zhang \cite{MR4271962} and the improvement to Gevrey class 2 by Wang-Wang-Zhang \cite{2021arXiv210300681Y}. We also mention  some other related works \cite{MR3461362, MR3710703, MR3464051}.  For the Prandtl system with 
a suitable background magnetic field, the local solutions in Sobolev or Gevrey function spaces were obtained  in  \cite{MR3882222,MR4342301,MR4102162,MR4270479}, and the global   analytic solution was established  recently in  \cite{ MR4213671,MR4293727}.    
Note that all of these global-in-time existence results are in 2D setting under  some suitable structural conditions on the initial data.   Hence, the  global-in-time properties of these systems in 3D setting remain  unknown.

On the other hand, there have been some recent studies on the global well-posedness of   hydrostatic Navier-Stokes equations (also called hydrostatic Prandtl equations); these  can be used to describe a large scale motion in  atmospheric and oceanic sciences, and  are derived  as
 a   limit  of the Navier-Stokes equations in a very thin domain.  The hydrostatic Navier-Stokes equations  have the same degeneracy structure  as  in   the classical Prandtl system, 
 so that  the analytic regularity is sufficient to obtain the  local well-posedness theory for general initial data without any structural assumption. Moreover, due to the damping effect, by combining  the  vertical diffusion and the Poincar\'e inequality in the vertical interval,   the global-in-time property of the hydrostatic Navier-Stokes equations  was recently verified  by  Paicu-Zhang-Zhang \cite{MR4125518}.   However,   differently from the classical Prandtl equations, the analytic regularity is necessary for the well-posedness theory of  the hydrostatic Prandtl equations without any structural assumption (cf. Renardy's work \cite{MR2563627}).  Hence,  in order to investigate the well-posedness theory in a larger function space  than  analytic,  some  structural conditions are required. We mention that the  well-posedness in Sobolev space
 of  the hydrostatic Navier-Stokes equations remains   unknown. Under the convex assumption,  Masmoudi-Wong \cite{MR2898740} established the well-posedness in a Sobolev space
   of the hydrostatic  Euler equations which is  the inviscid version of the hydrostatic  Navier-Stokes equation. However, under the same   convex assumption, the    well-posedness has been proved  for the hydrostatic  Navier-Stokes equation only in the Gevrey function space by
G\'{e}rard-Varet-Masmoudi-Vicol in \cite{MR4149066}, with
 the Gevrey index  being up to $9/8$, and in \cite{arXiv:2205.15829, 2021arXiv2206.03873} up to  $3/2$ that is believed to be optimal.  Furthermore,   the global well-posedness theory of the hyperbolic version  of the 2D hydrostatic Navier-Stokes equation   was obtained recently by Aarach   \cite{2021arXiv211113052A}  and  Paicu-Zhang \cite{2021arXiv211112836P} in analytic and Gevrey function spaces, respectively, with the  extension to the 3D case in   \cite{2022arXiv220409173L}.   A similar well-posedness 
 result on the hyperbolic Prandtl equations in Gevrey class was proven in \cite{2021arXiv211210450L}.   These  results  imply that  the hyperbolic feature may lead to 
 some kind of stability effect  compared with the parabolic counterparts. 
 
 This work aims to combine  some kind of  intrinsic hyperbolic structure   with the extra magnetic damping term   to derive the global-in-time well-posedness for general Gevrey classes with sharp index 2  initial data without any structural assumption.  More precisely, inspired by \cite{lmy,MR3925144},  we introduce some  auxiliary functions to cancel the nonlocal terms involving the loss of tangential derivatives, and then investigate the intrinsic hyperbolic feature for   evolution equations of the auxiliary functions.  %This enables us to extend  the Gevrey index to the sharp one $2$.  

In the   discussion that follows, we assume that  $ \boldsymbol{ U}= \boldsymbol{ 0}$ and  that $\partial_xp= \boldsymbol{ 0}$. We expect the approach  can be applied to the case with a more general  outer flow satisfying some  suitable conditions;  that is, we consider    system \eqref{dpe+} as
\begin{equation}\label{dpe}
	\left\{
\begin{aligned}
&(\partial_t   +  \boldsymbol{u}\cdot\partial_x +v\partial_{y} -\partial_{y}^2) \boldsymbol{u} +  \boldsymbol{u}=0,\\
&\partial_x\cdot  \boldsymbol{u}+\partial_{y}v=0,\\
&( \boldsymbol{u}, v)|_{y=0}=( \boldsymbol{0}, 0),  \quad \lim_{y\rightarrow+\infty}  \boldsymbol{u}= \boldsymbol{ 0},\\
& \boldsymbol{u}|_{t=0}= \boldsymbol{u}_0.
\end{aligned}
\right.
\end{equation}
Note \eqref{dpe} is a degenerate parabolic system with a loss of tangential derivatives in the nonlocal normal velocity given by
\begin{eqnarray*}
	v(t,x,y)=-\int_0^y\partial_x\cdot \boldsymbol{u}(t,x,\tilde y)d\tilde y 
\end{eqnarray*}
as     the main degeneracy feature of  the     Prandtl type equations.  

\medskip
 \noindent {\bf Notations.}  In what follows, we will use  $\norm{\cdot}_{L^2}$ and $\inner{\cdot, \cdot}_{L^2}$ to denote the norm and inner product of  $L^2=L^2(\mathbb R_+^n)$,   and use the notation   $\norm{\cdot}_{L_x^2}$ and $\inner{\cdot, \cdot}_{L_x^2}$  when the variable $x$ is specified. Similar notation  will be used for $L^\infty$. In addition, we use $L^p_x L^q_y = L^p (\mathbb R^{n-1}; L^q(\mathbb R_+))$ for the classical Sobolev space.  For a vector-valued function $A=(A_1,A_2, \ldots, A_k)$, we use  the convention that $\norm{A}_{L^2}^2=\sum_{1\leq j\leq k}\norm{A_j}_{L^2}^2$. Throughout the paper,   $\comi y:=(1+y^2)^{1/2}$.
 
   \begin{definition} 
\label{defgev} 	Let $\ell>1/2$   be a given number and let $\tau=\tau_N$ be a given weight function defined by 
 \begin{equation}\label{tauweigh}
 \tau(y)=	\tau_N(y)=\big(N+y^2\big)^{1/2},
 \end{equation}
 with $N\geq 1$ a fixed integer depending only on $\ell$ such that
 \begin{equation}\label{rel:Nell}
 \frac{\ell+3}{N^{1/2}}+  \frac{\ell^2+\ell}{N}\leq \frac{1}{8}. 
\end{equation}
The space $X_{r}$ of  Gevrey functions with the Gevrey index $2$ consists of all  smooth  (scalar or vector-valued) functions
 $h$  such that the  norm  $\norm{h}_{X_{r}}<+\infty,$  where  %  $\abs{\cdot}_{X_\rho}$ is defined by
 \begin{eqnarray*}
\begin{aligned}
	 \norm{h}_{X_r  }^2= & \sum_{0\leq j \leq 3}  \sum_{ \alpha\in\mathbb Z_+^{n-1}}  L_{r,\abs{\alpha}+j}^2 \norm{\tau^{\ell+j}\partial_{x}^\alpha \partial_y^j  h }_{L^2}^2,
	 \end{aligned}
\end{eqnarray*}
with   
 \begin{equation} \label{radiu}
  	L_{r,k}=
  	\frac{r^{k+1}(k+1)^{10}}{ (k!)^2}, \quad k\in\mathbb Z_+, \ r>0.
   \end{equation}
 % and  we omit the time dependence of $\rho$ if no confusion occurs.  
%Moreover we will use the convention that 
  \end{definition}

\begin{theorem}\label{thmmagprandlt} Let the dimension be $n=2$ or $3$, and let
  the initial datum $ \boldsymbol{u}_0$ in \eqref{dpe} belong  to $X_{2\rho_0}$ for some $ \rho_0>0$, compatible with the boundary condition in \eqref{dpe}.   Suppose that 
  \begin{eqnarray*}
  	\norm{ \boldsymbol{u}_0}_{X_{2\rho_0}}\leq \eps_0
  \end{eqnarray*}
for some sufficiently  small $\eps_0>0$. 	
	  Then  the magnetic Prandtl model \eqref{dpe} admits a unique global solution $ \boldsymbol{u}\in L^\infty\inner{[0,+\infty[; X_{\rho}}$ with
	\begin{equation}\label{derho}
		\rho=\rho(t)=\frac{\rho_0}{2}+\frac{\rho_0}{2}e^{-t/12}.
	\end{equation}
	Moreover,
	\begin{eqnarray*}
	\sup_{t\geq 0}	e^{t/4} \norm{ \boldsymbol{u}(t)}_{X_{\rho(t)}}+\left(\int_0^{+\infty}e^{t/2} \norm{\partial_y  \boldsymbol{u}(t)}_{X_{\rho(t)}} ^2dt \right)^{1/2} \leq \frac{4(\rho_0+1)}{\rho_0}  \eps_0. 
\end{eqnarray*}
\end{theorem}

 \begin{remark} Note that the same argument shows that 
  	 the global well-posedness property  holds  for  a Gevrey function space with the Gevrey index in the interval $[1, 2]$. 
  \end{remark}

\section{A priori estimate in 2D}\label{sec2}
To have a clear presentation,  we give a detailed proof of  Theorem \ref{thmmagprandlt} in 2D.   For $n=2$, we have the scalar tangential velocity $u$, and 
then the  magnetic Prandtl model \eqref{dpe} can be written as
\begin{equation}
	\label{2dprandtl}
		\left\{
\begin{aligned}
&(\partial_t   +u\partial_x +v\partial_{y} -\partial_{y}^2)u + u=0,\\
&\partial_xu+\partial_{y}v=0,\\
&(u, v)|_{y=0}=(  0, 0),  \quad \lim_{y\rightarrow+\infty} u=0 ,\\
&  u |_{t=0}= u_0.
\end{aligned}
\right.
\end{equation}
 The key part  is to derive  an {\it  a  priori} estimate for   \eqref{2dprandtl}  so that  the existence and uniqueness stated in Theorem \ref{thmmagprandlt}   follows from  a standard argument.  Hence,  for brevity, we only present the proof of the following {\it a priori} estimates for solutions to \eqref{2dprandtl} with Gevrey regularity:  

\begin{theorem}
	\label{th:ape}
	Let $X_\rho$ be the Gevrey space given in Definition \ref{defgev}. 
	Suppose that the initial datum $u_0$ in \eqref{2dprandtl} belongs to  $X_{2\rho_0}$ for some $ \rho_0>0$, and let $u\in L^\infty \inner{[0,+\infty[;\  X_{ \rho}}$ be any solution to \eqref{2dprandtl} satisfying that 
     \begin{equation}\label{apasu}
  \int_0^\infty \big(\norm{u(t)}_{X_{\rho(t)}}^2+\norm{\partial_yu(t)}_{X_{\rho(t)}}^2\big) dt<+\infty, 
  \end{equation}
where  $\rho$ is defined by  \eqref{derho}. 
 If 
	 \begin{equation}
	    	\label{smallness}
	    	 \norm{u_0}_{X_{2\rho_0}}\leq \varepsilon_0
	    \end{equation}
for some sufficiently small $\varepsilon_0>0,$ then
\begin{eqnarray*}
	\sup_{t\geq 0}	e^{t/4} \norm{ u(t)}_{X_{\rho(t)}}+\left(\int_0^{+\infty}e^{t/2} \norm{\partial_y   u(t)}_{X_{\rho(t)}} ^2dt\right)^{1/2}  \leq \frac{4(\rho_0+1)}{\rho_0}  \eps_0. 
\end{eqnarray*}
\end{theorem}

\subsection{Methodologies and   auxiliary functions} The main difficulty for the well-posedness of Prandtl type equations comes from the    loss of  tangential derivatives.  To overcome the  tangential  degeneracy,  the abstract Cauchy-Kowalewski Theorem is an effective method, for example, 
for the local  existence and uniqueness in an analytic setting, cf.  \cite{MR1617542}.  
 However, it is  not trivial to relax the analyticity regularity to a larger function space such as
  Gevrey space for well-posedness. For this,   some intrinsic  structure of the system needs to be used, cf. \cite{MR3925144, lmy}.   As was shown in \cite{2022arXiv220409173L,2021arXiv211112836P,2021arXiv211210450L},   the well-posedness is well expected in the Gevrey class rather than the analytic setting for the hyperbolic  Prandtl equations without any structural assumption. This indicates that  the   hyperbolic feature may act as a stabilizing factor for the Gevrey well-posedness theory.   In this paper,   with the extra damping term in magnetic Prandtl model,    we will prove the global-in-time existence in a Gevrey function space by exploring the intrinsic hyperbolic feature for auxiliary functions. 
    
 In order to clarify the stability effect of  the hyperbolic feature,  let us use   the following toy model  with a hyperbolic factor $\partial_t^2$ to illustrate   the main idea in the proof:
\begin{equation} \label{toymo+1}
	\partial_t^2 h+ h\partial_xh-\partial_y^2h=0, \quad h|_{t=0}=h_0 \ \textrm{ and }\ \partial_t h|_{t=0}=h_1,
\end{equation}
where  one order $x$-derivative is lost  in twice-time differentiation.  
  Denoting that $g=\partial_th$, the above Cauchy problem can be rewritten as
\begin{equation}
	\label{toymodel}
	\left\{
	\begin{aligned}
	&	\partial_t h=g,\\
	&	\partial_tg-h\partial_xh-\partial_y^2h=0,\\
	&(h,g)|_{t=0}=(h_0,h_1).
	\end{aligned}
	\right. 
\end{equation}  
To overcome the loss of the $x$-derivative,  we introduce a decreasing function of radius$\rho=\rho(t)$. In what follows, $\rho$ depends on time $t$, but we only write it as $\rho$ instead of $\rho(t)$ for  simplicity of the notations. Moreover,  we denote by $\rho'$ and $\rho''$ the first and the second time derivatives of   $\rho,$ respectively.   
Now we derive   estimates on the   Gevrey  norm 
\begin{eqnarray*}
\sum_{m=0}^{+\infty}	\frac{\rho^{m+1}}{m!^2}\inner{\norm{\partial_x^mh}_{L^2}+\norm{ \partial_x^m \partial_yh}_{L^2}+\norm{ \partial_x^m g}_{L^2 }},
\end{eqnarray*}
where  $(h,g)$ solves the Cauchy problem \eqref{toymodel}. 
By direct  calculation and using the fact that  $\rho'\leq 0,$ we have that
\begin{align*}
&	\frac{1}{2}\frac{d}{dt}\bigg(\frac{\rho^{2(m+1)}}{m!^4}\Big(\norm{\partial_x^mh}_{L^2}^2+\norm{ \partial_x^m \partial_yh}_{L^2}^2+\norm{ \partial_x^m g}_{L^2}^2\Big)\bigg)\\
&=\rho'\frac{m+1}{\rho}\frac{\rho^{2(m+1)}}{m!^4}\inner{\norm{\partial_x^mh}_{L^2}^2+\norm{ \partial_x^m \partial_yh}_{L^2}^2+\norm{ \partial_x^m g}_{L^2}^2} +\frac{\rho^{2(m+1)}}{m!^4}  \big(h \partial_x^{m+1} h,\  \partial_x^m g\big)_{L^2}  +\textrm{ l.o.t.}\\
&\leq \rho'\frac{m+1}{\rho}\frac{\rho^{2(m+1)}}{m!^4}  \inner{ \norm{ \partial_x^m \partial_yh}_{L^2}^2+\norm{ \partial_x^m g}_{L^2}^2} \\
&\quad+  \frac{(m+1)^2}{\rho}\norm{h}_{L^\infty}\bigg(\frac{\rho^{  m+2}}{(m+1)!^2} \norm{\partial_x^{m+1} h}_{L^2}\bigg) \bigg(\frac{\rho^{ m+1}}{m!^2} \norm{  \partial_x^m g}_{L^2} \bigg)+\textrm{ l.o.t.},
\end{align*}
 where l.o.t. refers to the lower-order terms that can be  controlled directly.  Moreover, we have ( cf.  Subsection  \ref{subseaux} for details) that
\begin{eqnarray*}
	 \rho'\frac{m+1}{\rho}\frac{\rho^{2(m+1)}}{m!^4}  \norm{ \partial_x^m g}_{L^2}^2\leq \rho'^3 \frac{(m+1)^3}{\rho^3}\frac{\rho^{2(m+1)}}{m!^4}  \norm{\partial_{x}^{m}h}_{L^2}^2 +\textrm{ l.o.t.}.
	 \end{eqnarray*}
In summary, 
	\begin{eqnarray*}
	\begin{aligned}
& \frac{1}{2}\frac{d}{dt}\bigg(\frac{\rho^{2(m+1)}}{m!^4}\inner{\norm{\partial_x^mh}_{L^2}^2+\norm{\partial_x^m\partial_yh}_{L^2}^2+\norm{ \partial_x^mg}_{L^2}^2}\bigg)\\
&\qquad  \leq \frac{1}{2}\rho'\frac{m+1}{\rho}\frac{\rho^{2(m+1)}}{m!^4}  \inner{ \norm{ \partial_x^m \partial_yh}_{L^2}^2+\norm{ \partial_x^m g}_{L^2}^2}+\frac12 \rho'^3 \frac{(m+1)^3}{\rho^3}\frac{\rho^{2(m+1)}}{m!^4}  \norm{\partial_{x}^{m}h}_{L^2}^2\\
&\qquad\qquad+ \frac{1}{2}\norm{h}_{L^\infty}\frac{ m+1 }{\rho} \frac{\rho^{2( m+1)}}{m!^4} \norm{  \partial_x^m g}_{L^2}^2 +\frac{1}{2}\norm{h}_{L^\infty}\frac{(m+1)^3}{\rho^3} \frac{\rho^{ 2( m+2)}}{(m+1)!^4} \norm{\partial_x^{m+1} h}_{L^2}^2+\textrm{ l.o.t.}. 
\end{aligned}
\end{eqnarray*}
If we define a norm $\normm{\cdot}_{Y_\rho}$ by
\begin{equation}\label{trinorm}
	\normm{h}_{Y_\rho}^2=\sum_{m= 0}^{+\infty}\frac{m+1}{\rho}\frac{\rho^{2(m+1)}}{m!^4}  \inner{ \norm{ \partial_x^m \partial_yh}_{L^2}^2+\norm{ \partial_x^m \partial_th}_{L^2}^2} + \sum_{m= 0}^{+\infty}\frac{(m+1)^3}{\rho^3}\frac{\rho^{2(m+1)}}{m!^4}  \norm{\partial_{x}^{m}h}_{L^2}^2,
\end{equation}   
then it follows from the above inequality that  \begin{multline*}
	 \frac{1}{2}\frac{d}{dt}\bigg(\sum_{m= 0}^{+\infty}\frac{\rho^{2(m+1)}}{m!^4}\inner{\norm{\partial_x^mh}_{L^2}^2+\norm{ \partial_x^m \partial_yh}_{L^2}^2+\norm{ \partial_x^mg}_{L^2}^2}\bigg)\\
	 \leq \frac{1}{2}\Big( \max\big\{\rho', \rho'^3\big\}+\norm{h}_{L^\infty}\Big)\normm{h}_{Y_\rho}^3 +\textrm{ l.o.t.}\leq \textrm{ l.o.t.},
\end{multline*}
 provided  that $\norm{h}_{L^\infty}$ is bounded by  $|\max\big\{\rho', \rho'^3\big\}|$ because $\rho'<0$. 
 
Differently from  the hyperbolic toy model \eqref{toymo+1}, the magnetic Prandtl model is a parabolic initial-boundary problem. 
 If we perform  estimates for the 2D magnetic Prandtl model \eqref{2dprandtl} directly,   the above
 energy estimate cannot   be closed in the Gevrey norm $\norm{\cdot}_{X_\rho}$.   In order to overcome the loss of  derivative difficulty, as in \cite{MR3925144, lmy},  some anxilliary functions are needed. More precisely, 
for a solution   $u\in L^\infty \inner{[0,+\infty[;\  X_{ \rho}}$ to \eqref{2dprandtl} satisfying the conditions in Theorem \ref{th:ape}, let 
 $\mathcal U$  be  a solution to the   problem 
\begin{eqnarray}\label{mau}
\left\{
\begin{aligned}
& \big(\partial_t+u\partial_x +v\partial_y-\partial_y^2\big)    \int_0^y\mathcal U  d\tilde y+\int_0^y\mathcal U  d\tilde y  =  -\partial_x^3 v,\\
& \mathcal U|_{t=0}=0, \quad \partial_y\mathcal U|_{y=0}=\mathcal U|_{y\rightarrow+\infty}=0.
\end{aligned}
\right.
\end{eqnarray} 
The existence of $\mathcal U$  follows from the standard parabolic theory. In fact, one can
 first apply the existence theory for  linear parabolic equations to  construct a solution $f$ to the   initial-boundary problem
\begin{eqnarray}\label{f}
	\left\{
\begin{aligned}
& \big(\partial_t+u\partial_x +v\partial_y-\partial_y^2\big)    f+f=  -\partial_x^3 v,\\
&f|_{t=0}=0, \quad f|_{y=0}=\partial_yf|_{y\rightarrow+\infty}=0,
\end{aligned}
\right.
\end{eqnarray}
and then set $f=\int_0^y \mathcal U(t,x,\tilde y) d\tilde y$. Moreover, under condition \eqref{apasu},   we use the parabolic regularity theory   to conclude that, for $\ell>1/2 $ and for any $m\geq0,$
\begin{equation}\label{reg:mu}
\left\{
\begin{aligned}
	 & \comi y^{-\ell} \partial_x^m\int_0^y \mathcal U(t,x,\tilde y) d\tilde y =\comi y^{-\ell}\partial_x^m f\in L^2\inner{[0,+\infty[; L^2},\\
	& \partial_x^m  \mathcal U = \partial_y\partial_x^m f\in L^2\inner{[0,+\infty[; L_x^2  H_y^1 }.
	\end{aligned}
	\right.
\end{equation}
%where
%$
%	\comi y=(1+y^2)^{1/2}. 
%$
	The above auxiliary functions are slightly different from those introduced
	in  \cite{lmy} which were inspired by  \cite{MR3925144}.  As in \cite{lmy}. by  $\mathcal U$ and 
 \begin{eqnarray}\label{ma}
  \lambda=\partial_x^3u- (\partial_yu)\int_0^y\mathcal U(t,x,\tilde y)  d\tilde y,
\end{eqnarray}  
we can cancel the  term involving  $v$ with  the  highest tangential derivative
as shown in  Subsection \ref{subseaux}.  The two auxiliary functions 
have the   relation
 \begin{eqnarray}\label{ymau}
\begin{aligned}
\big(\partial_t+u\partial_x +v\partial_y-\partial_y^2\big) \mathcal U +\mathcal U  
=&\partial_x\lambda +(\partial_x\partial_yu)\int_0^y\mathcal U(t,x,\tilde y)  d\tilde y+(\partial_xu)\mathcal U.
\end{aligned}
\end{eqnarray}
%which follows by applying $\partial_y$ to the first equation in \eqref{mau}.  Moreover
In addition, 
\begin{equation}\label{lateq}
	\big(\partial_t+u\partial_x +v\partial_y-\partial_y^2\big) \partial_x \lambda +\partial_x\lambda 
= -4(\partial_xu)\partial_x^4 u-3(\partial_x^3v)+\textrm{ l.o.t.}.
\end{equation}
Recall that the main structure of  \eqref{toymo+1} or \eqref{toymodel} is that  the loss of the one order $x$-derivative occurs  in the twice-time differentiation. 
If   $\mathcal U$  behaves like the 3-order $x$-derivative of $u$,  then
  \eqref{ymau}-\eqref{lateq}  admit a similar structure as shown in    toy model \eqref{toymodel}. More precisely, we have the loss of one order $x$-derivative in the twice application of the  Prandtl operator instead of a time differentiation so that the pair $(\mathcal U, \partial_x\lambda)$    plays a similar role as $(h, g)$ in \eqref{toymodel}.      Inspired by the triple norm defined in \eqref{trinorm},  we define the   norms on the solutions and the auxiliary functions as follows:

\begin{definition}\label{defabnorm}
Let $\mathcal U, \lambda$ be defined as in \eqref{mau} and \eqref{ma}, respectively. 	By denoting
	\begin{equation*}
		\vec a=(u,  \mathcal U,   \lambda ),
	\end{equation*} 
	we define $|\vec a|_{X_\rho}, |\vec a|_{Y_\rho}$ and $ |\vec a|_{Z_\rho} $ by 
	\begin{eqnarray*}
	\begin{aligned}
		|\vec a|_{X_\rho}^2&= \norm{u}_{X_\rho}^2+\sum_{m=0}^{+\infty}L_{\rho,m+2}^2\norm{\partial_x^m\lambda }_{L^2}^2+\sum_{m=0}^{+\infty}L_{\rho,m+3}^2\norm{\partial_x^m\mathcal U}_{L^2}^2,\\
		|\vec a|_{Y_\rho}^2&=\sum_{0\leq j\leq 3} \sum_{m=0}^{+\infty}\frac{m+j+1}{\rho} L_{\rho,m+j}^2 \norm{\tau^{\ell+j}\partial_x^m\partial_y^ju}_{L^2}^2+\sum_{m=0}^{+\infty} \frac{m+3}{\rho}L_{\rho,m+2}^2\norm{\partial_x^m\lambda }_{L^2}^2\\
		&\quad +\sum_{m=0}^{+\infty} \frac{(m+4)^3}{\rho^3}L_{\rho,m+3}^2\norm{\partial_x^m\mathcal U}_{L^2}^2,\\
		|\vec a|_{Z_\rho}^2&= \norm{\partial_y u}_{X_\rho}^2+\sum_{m=0}^{+\infty}L_{\rho,m+2}^2\norm{\partial_y\partial_x^m\lambda }_{L^2}^2+\sum_{m=0}^{+\infty}L_{\rho,m+3}^2\norm{\partial_y\partial_x^m\mathcal U}_{L^2}^2,
		\end{aligned}
	\end{eqnarray*}
	where $\norm{\cdot}_{X_\rho}$ and $L_{\rho,k}$ are given as in Definition \ref{defgev}. 
\end{definition}

\begin{remark}
	The norms defined above satisfy  that 
	\begin{eqnarray*}
		\norm{u}_{X_\rho}\leq |\vec a|_{X_\rho}\ \textrm{ and }\ \norm{\partial_yu}_{X_\rho}\leq |\vec a|_{Z_\rho}.
	\end{eqnarray*}
	If $\rho\leq 1$, then
		\begin{eqnarray*}
		  |\vec a|_{X_\rho}\leq |\vec a|_{Y_\rho}.
	\end{eqnarray*}
\end{remark}

In view of the above remark,  the {\it a priori} estimate given in Theorem \ref{th:ape} holds from the
following theorem:

\begin{theorem}\label{thmapri}
Under the same assumption as  in Theorem \ref{th:ape} we have that%,  with the notations in Definition \ref{defabnorm},
   \begin{eqnarray*}
		\sup_{t\geq 0} e^{t/4}|\vec a(t)|_{X_{\rho(t)}}+\Big(\int_0^{+\infty}e^{t/2}|\vec a(t)|_{Z_{\rho(t)}}^2dt\Big)^{1/2}\leq \frac{4 (\rho_0+1) }{\rho_0} \eps_0.
	\end{eqnarray*}   
\end{theorem}

The proof of this theorem is given in  Subsections  \ref{subseaux}-\ref{subsec:com}.   
 To simplify the notations, we will use  $C$ to denote a generic constant  which may vary from line to line and  depend only on  the Sobolev embedding constants and the constants  $ \ell $ and  $\rho_0$ in Definition \ref{defgev} and \eqref{derho}.  
Observe that $X_{r_1} \subset X_{r_2} $ for $r_1\geq r_2$.  Then
 we can assume, without loss of generality,  that the initial radius $\rho_0\leq 1$.  We now
  list some   
 facts that follow  directly  from the definition \eqref{derho} of  $\rho$.  For  $t\geq 0,$
  \begin{equation}
 	\label{eqvi}
   \rho_0/2\leq \rho(t)\leq \rho_0\leq 1, \quad \rho'(t)\leq\rho'^3<0, \quad \rho''(t)-\frac{\rho'(t)^2}{\rho(t)}= \frac{\rho_0   e^{-t/12}}{288(1+e^{-t/12})} \geq 0.
 \end{equation} 
 Recall  that $L_{\rho,m}$ is defined in \eqref{radiu}. Then
 \begin{equation}\label{dev:rho}
\forall\ m\geq 0,\quad  	\frac{d}{dt}L_{\rho,m}=\rho'\frac{m+1}{\rho}L_{\rho,m}.
 \end{equation}   
We will use     the following        Young's inequality for   discrete convolution:
 \begin{eqnarray}\label{dis}
\bigg[ \sum_{m=0}^{\infty} \Big(\sum_{j=0}^m p_jq_{m-j}\Big)^2\bigg]^{1/2}\leq \Big(\sum_{m=0}^{\infty}   q_m^2 \Big)^{1/2} \sum_{j=0}^{\infty}   p_j.
\end{eqnarray}
Here    $ \set{p_{j}}_{j\geq0} $ and $\set {q_{j}}_{j\geq0} $ are positive sequences. As an immediate consequence of \eqref{dis},
\begin{equation}\label{conv}
	  \sum_{m=0}^{\infty} \sum_{j=0}^m p_jq_{m-j} r_m \leq \Big(\sum_{m=0}^{\infty}   q_m^2 \Big)^{1/2}\Big(\sum_{m=0}^{\infty}   r_m^2 \Big)^{1/2}\sum_{j=0}^{\infty}   p_j
\end{equation}
holds  for any positive sequences $ \set{p_{j}}_{j\geq0}, \set {q_{j}}_{j\geq0}$ and $\set {r_{j}}_{j\geq0}. $ 

 \subsection{Estimate on the auxiliary functions}\label{subseaux}
 
 The following proposition  is the main part of the proof for Theorem: \ref{thmapri}.% devoted to   proving the following
 
 \begin{proposition}\label{prpaux}
 Under the same assumption as that in Theorem \ref{th:ape},   we have that
 	\begin{eqnarray*}
 		\begin{aligned}
	& \frac12\frac{d}{dt}\sum_{m=0}^{+\infty} \bigg(L_{\rho,m+2}^2\norm{\partial_x^{m} \lambda}_{L^2}^2+L_{\rho,m+3}^2\norm{\partial_x^{m} \mathcal U}_{L^2}^2 \bigg) +\sum_{m=0}^{+\infty}\Big(L_{\rho,m+2}^2\norm{                          \partial_x^{m}\partial_y \lambda}_{L^2}^2+L_{\rho,m+3}^2\norm{                          \partial_x^{m}\partial_y \mathcal U}_{L^2}^2\Big)\\
	&\quad+\frac12\sum_{m=0}^{+\infty}\Big(L_{\rho,m+2}^2\norm{\partial_x^{m} \lambda}_{L^2}^2+L_{\rho,m+3}^2\norm{                          \partial_x^{m}  \mathcal U}_{L^2}^2\Big) \\
	&\quad+\frac14\frac{d}{dt}\sum_{m=0}^{+\infty}     \rho'^2 \frac{(m+4)^2}{\rho^2}L_{\rho,m+3}^2\norm{   \partial_{x}^{m}\mathcal U}_{L^2}^2+ \frac12\sum_{m=0}^{+\infty} \rho'^2 \frac{(m+4)^2}{\rho^2}  L_{\rho,m+3} ^2\norm{ \partial_{x}^{m}\mathcal U }_{L^2}^2\\
	&\leq \frac{1}{4} \rho'^3\sum_{m=0}^{+\infty} \bigg(\frac{m+3}{\rho}L_{\rho,m+2}^2\norm{\partial_x^{m} \lambda}_{L^2}^2+\frac{(m+4)^3}{\rho^3}L_{\rho,m+3}^2\norm{\partial_x^{m} \mathcal U}_{L^2}^2\bigg) +C\inner{ |\vec a|_{X_\rho}+  |\vec a|_{X_\rho}^2}\inner{|\vec a|_{Y_\rho}^2+|\vec a|_{Z_\rho}^2}.
	\end{aligned}
 	\end{eqnarray*}
 %	recalling  the norms in the last line are given in Definition \ref{defabnorm}.
 \end{proposition}
 
%To derive the equation for $\lambda$,  we   apply $ \partial_x^3$ to the first equation in   \eqref{2dprandtl} and meanwhile multiply the first equation in \eqref{mau} by $\partial_y u$, and then subtract one by another  to cancel the worst term $(\partial_x^3v)\partial_yu$; this gives
First of all, the auxiliary function $\lambda$ satisfies the following equation:
\begin{multline*} 
\big(\partial_t+u\partial_x +v\partial_y-\partial_y^2\big)  \lambda +\lambda \\
= -4(\partial_xu)\partial_x^3 u-3(\partial_x^2v)\partial_x\partial_y u-3(\partial_x^2u)^2-3(\partial_xv)\partial_x^2\partial_y u +2(\partial_y ^2u) \mathcal U+(\partial_y u)\int_0^y\mathcal U  d\tilde y.
\end{multline*}
%Using \eqref{ma} as well as the fact $\partial_x^2v=-\int_0^y \partial_x^3 u(t,x,\tilde y)d\tilde y$, we can rewrite the above equation as
That is, 
\begin{equation}
	\label{llamd+}
	\big(\partial_t+u\partial_x+v\partial_y-\partial_y^2\big)  \lambda +\lambda=H,
\end{equation}
with
\begin{equation}\label{llamd++}
	\begin{aligned}
H&= -4(\partial_xu)\lambda  +\Big (\partial_y u -4(\partial_xu) \partial_yu \Big)\int_0^y\mathcal U  d\tilde y +3(\partial_x\partial_y u) \int_0^y\lambda(t,x,\tilde y) d\tilde y \\
&\quad +3(\partial_x\partial_y u) \int_0^y\bigg((\partial_yu)(t,x,\tilde y)\int_0^{\tilde y}\mathcal U(t,x,r)dr\bigg) d\tilde y -3(\partial_x^2u)^2-3(\partial_xv)\partial_x^2\partial_y u +2(\partial_y ^2u) \mathcal U.
	\end{aligned}
\end{equation}
By the definition \eqref{ma} of $\lambda$ and the fact that  $\lambda|_{y=0}=\lambda|_{y\rightarrow+\infty}=0,$ and the assumptions \eqref{apasu} and \eqref{reg:mu}, we have, % as well as  Sobolev inequality,  to conclude that, 
for all $m\geq 0,$  that
$$\partial_x^m \lambda\in L^2\inner{[0,+\infty[;
\  L^2\Big (\mathbb R_x ; H_0^{1}(\mathbb R_+)\Big)}  \textrm{ and } \partial_x^m\partial_y^2\lambda \in L^2\inner{[0,+\infty[; L^2\Big (\mathbb R_x ; H^{-1}(\mathbb R_+)\Big)},$$
and 
\begin{eqnarray*}
 \partial_x^m H,\  \ \partial_x^m (u\partial_x \lambda +v\partial_y\lambda) \in L^2\inner{[0,+\infty[; L^2} .
\end{eqnarray*}
This implies that
\begin{eqnarray*}
	\partial_t\partial_x^m \lambda \in L^2\inner{[0,+\infty[;  L^2\Big (\mathbb R_x; H^{-1}(\mathbb R_+)\Big)}. 
\end{eqnarray*}
Thus, % by the classical theory on Sobolev spaces involving time (cf. \cite[Theorem 3 in Section 5.9]{MR2597943}), we conclude for each $m\geq 0$,
\begin{eqnarray*}
	t\mapsto \norm{\partial_x^m \lambda(t)}_{L^2}^2
\end{eqnarray*}
is absolutely continuous on $[0,+\infty[$, and,  moreover,
\begin{eqnarray*}
	\frac{1}{2}\frac{d}{dt}  \norm{\partial_x^{m} \lambda}_{L^2}^2=\comi{\partial_t\partial_x^{m} \lambda, \ \partial_x^{m} \lambda},
\end{eqnarray*}
where   $\comi{\cdot,\ \cdot}$ stands for the   pairing   between $L^2\Big (\mathbb R_x; H^{-1}(\mathbb R_+)\Big)$ and $L^2\Big (\mathbb R_x; H_0^{1}(\mathbb R_+)\Big)$.
Hence, by \eqref{dev:rho}, 
\begin{eqnarray*}
	\frac{1}{2}\frac{d}{dt}  L_{\rho,m+2}^2 \norm{\partial_x^{m} \lambda}_{L^2}^2=\rho'\frac{m+3}{\rho}L_{\rho,m+2}^2\norm{\partial_x^{m} \lambda}_{L^2}^2+L_{\rho,m+2}^2 \comi{\partial_t\partial_x^{m} \lambda, \ \partial_x^{m} \lambda}.
\end{eqnarray*}
Consequently, we   apply  $\partial_x^m$ to  \eqref{llamd+} and then take the pairing with $\partial_x^m\lambda$ between $L^2\Big (\mathbb R_x; H^{-1}(\mathbb R_+)\Big)$ and $L^2\Big (\mathbb R_x; H_0^{1}(\mathbb R_+)\Big)$,   together with
the factor $L_{\rho,m+2}^2$,  to get that%and  take the summation for $m\geq 0 $; this gives  
 \begin{equation}\label{lambeq}
\begin{aligned}
	&\frac{1}{2}\frac{d}{dt}\sum_{m=0}^{+\infty}L_{\rho,m+2}^2\norm{\partial_x^{m} \lambda}_{L^2}^2+\sum_{m=0}^{+\infty}L_{\rho,m+2}^2\norm{                          \partial_x^{m}\partial_y \lambda}_{L^2}^2+\sum_{m=0}^{+\infty}L_{\rho,m+2}^2\norm{\partial_x^{m} \lambda}_{L^2}^2\\
	&= \rho'\sum_{m=0}^{+\infty} \frac{m+3}{\rho}L_{\rho,m+2}^2\norm{\partial_x^{m} \lambda}_{L^2}^2+\sum_{m=0}^{+\infty}L_{\rho,m+2}^2\inner{\partial_x^m H,\,  \partial_x^m\lambda}_{L^2},
	\end{aligned}
\end{equation}
where  $H$ is given in \eqref{llamd++}. 
 
 The next lemmas are for the estimate on the last term  on the right side of \eqref{lambeq}.

 \begin{lemma}\label{lem2} Let $|\vec a|_{X_\rho}, |\vec a|_{Y_\rho}$ and $|\vec a|_{Z_\rho}$ be given as in Definition \ref{defabnorm}. 
 Under the same assumption as in Theorem \ref{th:ape}, we have 	 that
 \begin{equation}\label{mete}
 		\sum_{m=0}^{+\infty}L_{\rho,m+2}^2\Big(\partial_x^m \Big[3(\partial_x\partial_y u) \int_0^y\bigg((\partial_yu)(t,x,\tilde y)\int_0^{\tilde y}\mathcal U(t,x,r)dr\bigg) d\tilde y \Big],\,  \partial_x^m\lambda\Big)_{L^2} \leq C |\vec a|_{X_\rho} ^2|\vec a|_{Y_\rho}^2.
 	\end{equation}
 	Similarly, it holds that
 	\begin{multline*}
 		\sum_{m=0}^{+\infty}L_{\rho,m+2}^2\Big(\partial_x^m \Big[2(\partial_y ^2u) \mathcal U+\Big (\partial_y u -4(\partial_xu) \partial_yu \Big)\int_0^y\mathcal U  d\tilde y \Big],\,  \partial_x^m\lambda\Big)_{L^2}\\
 		\leq C|\vec a|_{X_\rho}|\vec a|_{Z_\rho}^2+  C\big( |\vec a|_{X_\rho}+  |\vec a|_{X_\rho}^2\big)|\vec a|_{Y_\rho}^2.
 			\end{multline*}

 \end{lemma}

\begin{proof}
It suffices to prove the first estimate because the second one can be obtained similarly.

{\it Step 1.} For simplicity of notation,  set % we denote in the following discussion   
\begin{equation}\label{f}
	 \mathscr L (u,\mathcal U):=\int_0^y\bigg((\partial_yu)(t,x,\tilde y)\int_0^{\tilde y}\mathcal U(t,x,r)dr\bigg) d\tilde y. 
\end{equation}
We  first prove the following two estimates for $ \mathscr L (u,\mathcal U)$:
\begin{equation}\label{prest}
	\Big(\sum_{m=0}^{+\infty} \frac{ (m+4) ^{3}}{\rho^{3}} L_{\rho,m+3}^2  \norm{ \partial_x^{m} \mathscr L (u,\mathcal U)}_{L_x^2 L_y^\infty}^2\Big)^{1/2}\leq C|\vec a|_{X_\rho}  |\vec a|_{Y_\rho}
	\end{equation}
	and
	\begin{equation}\label{est:F}
			\Big(\sum_{m=0}^{+\infty}   L_{\rho,m+5}^2  \norm{ \partial_x^{m}\mathscr L (u,\mathcal U)}_{L^\infty}^2\Big)^{1/2}\leq C |\vec a|_{X_\rho}^2.
	\end{equation}
	  In fact, 
Leibniz's formula gives that
\begin{equation}\label{ambm}
	 \frac{ (m+4) ^{3/2}}{\rho^{3/2}} L_{\rho,m+3}  \norm{ \partial_x^{m} \mathscr L (u,\mathcal U)}_{L_x^2 L_y^\infty}\leq  a_m+b_m, 
\end{equation}
with
\begin{eqnarray*}
	\begin{aligned}
		 a_m&=C  \sum_{j=0}^{[m/2]}\frac{m!}{j!(m-j)!} \frac{ (m+4) ^{3/2}}{\rho^{3/2}} \frac{ L_{\rho,m+3} }{L_{\rho,j+3}L_{\rho,m-j+3}  }\\
		&\qquad\qquad\quad\times\big( L_{\rho,j+3} \norm{ \comi y^{\ell+1}\partial_x^{j}\partial_yu}_{L_x^\infty L_y^2} \big)\big(L_{\rho,m-j+3} \norm{   \partial_x^{m-j}\mathcal U}_{L^2}\big)\\
			b_m&= C  \sum_{j=[m/2]+1}^m\frac{m!}{j!(m-j)!} \frac{ (m+4) ^{3/2}}{\rho^{3/2}} \frac{ L_{\rho,m+3} }{L_{\rho,j+1}L_{\rho,m-j+5}  }\\
		&\qquad\qquad\qquad\times\big( L_{\rho,j+1} \norm{ \comi y^{ \ell+1}\partial_x^{j}\partial_yu}_{L^2} \big)\big(L_{\rho,m-j+5} \norm{   \partial_x^{m-j}\mathcal U}_{L_x^\infty L_y^2}\big),
	\end{aligned}
\end{eqnarray*}
 where  $[m/2]$ represents the largest integer less than or equal to $ m/2$.
Direct computation shows that, for any $0\leq j\leq [m/2]$, 
\begin{eqnarray*}
 \frac{m!}{j!(m-j)!} \frac{ (m+4) ^{3/2}}{\rho^{3/2}} \frac{ L_{\rho,m+3} }{L_{\rho,j+3}L_{\rho,m-j+3}  }\leq \frac{C}{ \rho^{4}}\frac{1}{ j+1 }\frac{(m-j+4)^{3/2}}{\rho^{3/2}}\leq \frac{C}{ j+1 }\frac{(m-j+4)^{3/2}}{\rho^{3/2}},
\end{eqnarray*}
where we have used the fact that $\rho_0/2\leq \rho\leq \rho_0$ % in \eqref{eqvi} 
in the last inequality. 
Then
\begin{eqnarray*}
\begin{aligned}
	\sum_{m=0}^{+\infty} a_m^2&\leq C\sum_{m=0}^{+\infty} \bigg[\sum_{j=0}^{[m/2]} 
		 \frac{ L_{\rho,j+3} \norm{ \comi y^{\ell+1}\partial_x^{j}\partial_yu}_{L_x^\infty L_y^2} }{j+1} \frac{(m-j+4)^{3/2}}{\rho^{3/2}}L_{\rho,m-j+3} \norm{   \partial_x^{m-j}\mathcal U}_{L^2}\bigg]^2\\
		 &\leq C\bigg( \sum_{j=0}^{+\infty} 
		 \frac{ L_{\rho,j+3} \norm{ \comi y^{\ell+1}\partial_x^{j}\partial_yu}_{L_x^\infty L_y^2} }{ j+1 }\bigg)^2 \sum_{j=0}^{+\infty} 
		\frac{(j+4)^{3}}{\rho^{3}}L_{\rho,j+3}^2 \norm{   \partial_x^{j}\mathcal U}_{L^2}^2\\
		 &\leq C|\vec a|_{X_\rho}^2 |\vec a|_{Y_\rho}^2,
		 \end{aligned}
\end{eqnarray*}
where we have used   Young's inequality \eqref{dis} for   discrete convolution. Similarly, by using the estimate
\begin{eqnarray*}
	\forall\  [m/2]\leq j\leq m,\quad	\frac{m!}{j!(m-j)!} \frac{ (m+4) ^{3/2}}{\rho^{3/2}} \frac{ L_{\rho,m+3} }{L_{\rho,j+1}L_{\rho,m-j+5}  }\leq  \frac{C}{m-j+1},
\end{eqnarray*}
we obtain, by  observing $|\vec a|_{X_\rho}\leq |\vec a|_{Y_\rho}$, that
\begin{eqnarray*}
	\sum_{m=0}^{+\infty} b_m^2\leq C\bigg( \sum_{j=0}^{+\infty} 
		\frac{L_{\rho,j+5}  \norm{   \partial_x^{j}\mathcal U}_{L_x^\infty L_y^2}}{ j+1} \bigg)^2 \sum_{j=0}^{+\infty} 
		  L_{\rho,j+1}^2 \norm{ \comi y^{ \ell+1}\partial_x^{j}\partial_yu}_{L^2}^2
		   \leq C|\vec a|_{X_\rho}^4 \leq C|\vec a|_{X_\rho}^2|\vec a|_{Y_\rho}^2.
\end{eqnarray*}
Combining the above estimates with \eqref{ambm} yields \eqref{prest}.  % By similar argument  we have \eqref{est:F}. 

 {\it Step 2}. To prove \eqref{mete}, % Recall $ \mathscr L (u,\mathcal U)$ is defined by \eqref{f}. 
   Leibniz's formula gives that
\begin{equation}\label{i1i2}
	\sum_{m=0}^{+\infty}L_{\rho,m+2}^2\Big(\partial_x^m \big[3(\partial_x\partial_y u)   \mathscr L (u,\mathcal U) \big],\,  \partial_x^m\lambda\Big)_{L^2}\leq I_1+I_2,
\end{equation}
where
\begin{eqnarray*}
	\begin{aligned}
		I_1&= 3 \sum_{m=0}^{+\infty}\sum_{j=0}^{[m/2]} {m\choose j} \frac{ L_{\rho,m+2}}{L_{\rho,j+4} L_{\rho,m-j+3}} \frac{\rho^{2}} {(m-j+4)^{3\over2}(m+1)^{1\over2}}L_{\rho,j+4} \norm{ \partial_x^{j+1}\partial_y u}_{L_x^\infty L_y^2}\\
		&\qquad          \times \Big(\frac{(m-j+4)^{3/2}}{\rho^{3/2}}L_{\rho,m-j+3}\norm{  \partial_x^{m-j} \mathscr L (u,\mathcal U)}_{L_x^2 L_y^\infty} \Big) \Big(\frac{(m+1)^{1/2}} {\rho^{1/2}}L_{\rho,m+2}\norm{ \partial_x^m\lambda}_{L^2}\Big),\\
		I_2&=3  \sum_{m=0}^{+\infty}\sum_{j= [m/2]+1}^m {m\choose j}\frac{L_{\rho,m+2}}{L_{\rho,j+2} L_{\rho,m-j+5}} L_{\rho,j+2} \norm{\partial_x^{j+1}\partial_y u}_{L^2}\\
		&\qquad  \qquad   \qquad   \qquad    \times \big(L_{\rho,m-j+5}\norm{ \partial_x^{m-j} \mathscr L (u,\mathcal U)}_{L^\infty} \big) \big(L_{\rho,m+2}\norm{ \partial_x^m\lambda}_{L^2}\big).
		\end{aligned}
	\end{eqnarray*}
Straightforward calculation shows that
\begin{eqnarray*}
	\begin{aligned}
\forall\ 0\leq j\leq [m/2],\quad 		\frac{m!}{j!(m-j)!} \frac{L_{\rho,m+2}}{L_{\rho,j+4} L_{\rho,m-j+3}} \frac{\rho^{2}} {(m-j+4)^{3/2}(m+1)^{1/2}} \leq  \frac{C}{ j+1}.
	\end{aligned}
\end{eqnarray*}
Thus, 
\begin{eqnarray*}
\begin{aligned}
	I_1& \leq C\sum_{m=0}^{+\infty}\sum_{j=0}^{[m/2]} \frac{\ L_{\rho,j+4} \norm{\partial_x^{j+1}\partial_y u}_{L_x^\infty L_y^2}}{ j+1}\\
		&\qquad \qquad         \times \Big(\frac{(m-j+4)^{3/2}}{\rho^{3/2}}L_{\rho,m-j+3}\norm{ \partial_x^{m-j} \mathscr L (u,\mathcal U)}_{L_x^2 L_y^\infty} \Big) \Big(\frac{(m+1)^{1/2}} {\rho^{1/2}}L_{\rho,m+2}\norm{ \partial_x^m\lambda}_{L^2}\Big)  \\
		& \leq C \sum_{j=0}^{+\infty} \frac{L_{\rho,j+4} \norm{ \partial_x^{j+1}\partial_y u}_{L_x^\infty L_y^2}}{j+1} \Big(\sum_{m=0}^{+\infty} \frac{ (m+4) ^{3}}{\rho^{3}} L_{\rho,m+3}^2  \norm{ \partial_x^{m} \mathscr L (u,\mathcal U)}_{L_x^2 L_y^\infty}^2\Big)^{1\over2} |\vec a|_{Y_\rho}\\
		&\leq C |\vec a|_{X_\rho}^2|\vec a|_{Y_\rho}^2,		 
\end{aligned}
\end{eqnarray*}
where we have used \eqref{conv} in the second inequality and \eqref{prest} in the last inequality.  

 For $I_2$, by   \eqref{est:F} and the estimate  \begin{eqnarray*}
	\begin{aligned}
\forall\   [m/2]+1\leq j\leq m,\quad 		 \frac{m!}{j!(m-j)!} \frac{L_{\rho,m+2}}{L_{\rho,j+2} L_{\rho,m-j+5}}  \leq \frac{C}{m-j+1},	
\end{aligned}
\end{eqnarray*}
we have that
\begin{eqnarray*}
	I_2\leq C \bigg(\sum_{j=0}^{+\infty}  L_{\rho,j+2} ^2\norm{ \partial_x^{j+1}\partial_y u}_{L^2} ^2\bigg)^{1/2} \bigg(\sum_{m=0}^{+\infty} \frac{L_{\rho,m+5}   \norm{ \partial_x^{m} \mathscr L (u,\mathcal U)}_{ L^\infty}}{m+1}  \bigg)  |\vec a|_{X_\rho}\leq C |\vec a|_{X_\rho}^4\leq C|\vec a|_{X_\rho}^2|\vec a|_{Y_\rho}^2.
\end{eqnarray*}
Substituting the above estimates on  $I_1$ and $I_2$ into  \eqref{i1i2} yields the first estimate of \eqref{mete} in Lemma \ref{lem2}.

{\it Step  3.} It remains to prove the second   estimate in Lemma \ref{lem2}. For this, write
\begin{equation}\label{2y}
	\begin{aligned}
		&	\sum_{m=0}^{+\infty}L_{\rho,m+2}^2\Big(\partial_x^m  \big[2(\partial_y ^2u) \mathcal U\big],\   \partial_x^m\lambda\Big)_{L^2}\\
			&\leq 2 \sum_{m=0}^{+\infty}\sum_{j=0}^{[m/2]}\frac{m!}{j!(m-j)!}\frac{L_{\rho,m+2}}{L_{\rho, j+5}L_{\rho,m-j+3}}\frac{\rho^2}{(m-j+4)^{3/2}(m+1)^{1/2}} (L_{\rho, j+5}\norm{\partial_x^j\partial_y ^2u}_{L^\infty})\\
			&\qquad\qquad\qquad\times \Big(\frac{(m-j+4)^{3/2}}{\rho^{3/2}}L_{\rho,m-j+3}\norm{\partial_x^{m-j}\mathcal U}_{L^2}\Big)\Big(\frac{(m+1)^{1/2}}{\rho^{1/2}}L_{\rho,m+2}\norm{  \partial_x^m\lambda}_{L^2}\Big)\\
			&\quad+ 2 \sum_{m=0}^{+\infty}\sum_{j=[m/2]+1}^m\frac{m!}{j!(m-j)!}\frac{L_{\rho,m+2}}{L_{\rho, j+2}L_{\rho,m-j+5}} (L_{\rho, j+2}\norm{\comi{y} ^{1/2}\partial_x^j\partial_y ^2u}_{L^2})\\
			&\qquad\qquad\qquad\times \big (L_{\rho,m-j+5}\norm{\partial_x^{m-j}\mathcal U}_{L^\infty}\big )\big (L_{\rho,m+2}\norm{ \partial_x^m\lambda}_{L^2}\big ).
	\end{aligned}
\end{equation} 
Similarly to the previous step,  we conclude that  the first term on the right hand side of \eqref{2y} is bounded from above by 
$C |\vec a|_{X_\rho} |\vec a|_{Y_\rho}^2$, and the last term is bounded from above by $C |\vec a|_{X_\rho}^2 |\vec a|_{Z_\rho} $.  Thus,
\begin{eqnarray*}
\begin{aligned}
		\sum_{m=0}^{+\infty}L_{\rho,m+2}^2\Big(\partial_x^m  \big[2(\partial_y ^2u) \mathcal U\big],\   \partial_x^m\lambda\Big)_{L^2}  		 \leq C|\vec a|_{X_\rho} |\vec a|_{Y_\rho}^2+ C|\vec a|_{X_\rho}^2|\vec a|_{Z_\rho}  \leq C |\vec a|_{X_\rho} |\vec a|_{Y_\rho}^2+ C|\vec a|_{X_\rho} |\vec a|_{Z_\rho}^2.
		\end{aligned} 
\end{eqnarray*}
Similarly,
\begin{eqnarray*}
		\sum_{m=0}^{+\infty}L_{\rho,m+2}^2\Big(\partial_x^m \Big[ \Big (\partial_y u -4(\partial_xu) \partial_yu \Big)\int_0^y\mathcal U  d\tilde y \Big],\,  \partial_x^m\lambda\Big)_{L^2} 
 		\leq  C \big( |\vec a|_{X_\rho}+  |\vec a|_{X_\rho}^2\big)|\vec a|_{Y_\rho}^2. 	
\end{eqnarray*}
Combining the above two estimates gives the second estimate in Lemma \ref{lem2}. This completes the proof of the lemma. 
\end{proof}

\begin{lemma}\label{lem1} Let $|\vec a|_{X_\rho}, |\vec a|_{Y_\rho}$ and $|\vec a|_{Z_\rho}$ be given as in Definition \ref{defabnorm}. 
 Under the same assumption as in Theorem \ref{th:ape}, we have that
 	\begin{eqnarray*}
 		\sum_{m=0}^{+\infty}L_{\rho,m+2}^2\Big(\partial_x^m \Big[3(\partial_x\partial_y u) \int_0^y\lambda(t,x,\tilde y) d\tilde y \Big],\,  \partial_x^m\lambda\Big)_{L^2}\leq C |\vec a|_{X_\rho} |\vec a|_{Y_\rho}^2,
 	\end{eqnarray*}
 	and
 	\begin{eqnarray*}
 		\sum_{m=0}^{+\infty}L_{\rho,m+2}^2\Big(\partial_x^m \big[-4(\partial_xu)\lambda-3(\partial_x^2u)^2-3(\partial_xv)\partial_x^2\partial_y u\big], \partial_x^m\lambda\Big)_{L^2}
 		\leq  C  |\vec a|_{X_\rho} |\vec a|_{Y_\rho}^2+C |\vec a|_{X_\rho}  |\vec a|_{Z_\rho}^2.
 	\end{eqnarray*}
 \end{lemma}
 
 \begin{proof} 
 	By Leibniz's formula, we have that
 	\begin{multline}\label{spl}
 		   \sum_{m=0}^{+\infty}L_{\rho,m+2}^2\Big(\partial_x^m \Big[(\partial_x\partial_y u) \int_0^y\lambda(t,x,\tilde y) d\tilde y \Big],\,  \partial_x^m\lambda\Big)_{L^2}\\
 		    \leq \sum_{m=0}^{+\infty}\sum_{j=0}^{[m/2]}{m\choose j} L_{\rho,m+2}^2\norm{\comi y^{\ell}   \partial_x^{j+1}\partial_y u}_{L_x^\infty L_y^2} \norm{\partial_x^{m-j}\lambda}_{L^2}\norm{\partial_x^m\lambda}_{L^2}\\
 		   +\sum_{m=0}^{+\infty}\sum_{j=[m/2]+1}^m{m\choose j} L_{\rho,m+2}^2\norm{\comi y^{\ell}   \partial_x^{j+1}\partial_y u}_{L^2} \norm{\partial_x^{m-j}\lambda}_{L_x^\infty L_y^2}\norm{\partial_x^m\lambda}_{L^2},
 		\end{multline}
 			where we have used the fact that  
 		\begin{eqnarray*}%\label{um}
 		 \norm{\comi{y}^{-\ell}\int_0^y\partial_{x}^{m-j}\lambda  d\tilde y}_{L_{x}^2 L_y^\infty}\leq C \norm{\partial_{x}^{m-j}\lambda }_{L^2}   		
 		\end{eqnarray*}
 	for $\ell>1/2$	  with $\comi y=(1+y^2)^{1/2}$.  
 		 By
 \begin{eqnarray*}
 	\forall\ 0\leq j\leq [m/2],\quad \frac{m!}{j!(m-j)!}\frac{L_{\rho,m+2}}{L_{\rho,j+4}L_{m-j+2}}\leq  \frac{C} { j+1  },
 \end{eqnarray*}
 we have that
 \begin{equation}\label{splfirs}
 	\begin{aligned}
 		&\sum_{m=0}^{+\infty}\sum_{j=0}^{[m/2]}{m\choose j} L_{\rho,m+2}^2\norm{\comi y^{\ell}   \partial_x^{j+1}\partial_y u}_{L_x^\infty L_y^2} \norm{\partial_x^{m-j}\lambda}_{L^2}\norm{\partial_x^m\lambda}_{L^2}\\
 		&\leq C\sum_{m=0}^{+\infty}\sum_{j=0}^{[m/2]}  \frac{L_{\rho,j+4} \norm{\comi y^{\ell}   \partial_x^{j+1}\partial_y u}_{L_x^\infty L_y^2}}{ j+1 }(L_{\rho,m-j+2} \norm{\partial_x^{m-j}\lambda}_{L^2})(L_{\rho,m+2} \norm{\partial_x^m\lambda}_{L^2})\\
 		&\leq C\sum_{j=0}^{+\infty}  \frac{L_{\rho,j+4} \norm{\comi y^{\ell}   \partial_x^{j+1}\partial_y u}_{L_x^\infty L_y^2}}{ j+1 }  \sum_{m=0}^{+\infty}L_{\rho,m+2}^2 \norm{\partial_x^{m}\lambda}_{L^2}^2 \leq C|\vec a|_{X_\rho}^3,   
 	\end{aligned}
 \end{equation}
 where \eqref{conv} is used in the second inequality  and 
 \begin{eqnarray*}
 	\sum_{j=0}^{+\infty}  \frac{L_{\rho,j+4} \norm{\comi y^{\ell}   \partial_x^{j+1}\partial_y u}_{L_x^\infty L_y^2}}{ j+1 } \leq C\Big(\sum_{j=0}^{+\infty}   L_{\rho,j+4}^2 \norm{\comi y^{\ell}   \partial_x^{j+1}\partial_y u}_{L_x^\infty L_y^2}^2\Big)^{1/2}\leq C |\vec a|_{X_\rho}
 \end{eqnarray*}
is used in the last inequality. 
 Similarly, by noting that
 \begin{eqnarray*}
 	\forall\   [m/2]+1\leq j\leq m,\quad \frac{m!}{j!(m-j)!}\frac{L_{\rho,m+2}}{L_{\rho,j+2}L_{m-j+4}} \leq   \frac{C}{ m-j+1  },
 \end{eqnarray*}
 we have that
 \begin{multline*}
 	 \sum_{m=0}^{+\infty}\sum_{j=[m/2]+1}^m{m\choose j} L_{\rho,m+2}^2\norm{\comi y^{\ell}   \partial_x^{j+1}\partial_y u}_{L^2} \norm{\partial_x^{m-j}\lambda}_{L_x^\infty L_y^2}\norm{\partial_x^m\lambda}_{L^2}\\
 	 \leq  C\bigg(\sum_{j=0}^{+\infty}   L_{\rho,j+2}^2 \norm{\comi y^{\ell}   \partial_x^{j+1}\partial_y u}_{L^2}^2\bigg)^{1\over2} \bigg(\sum_{m=0}^{+\infty}L_{\rho,m+2}^2 \norm{\partial_x^{m}\lambda}_{L^2}^2\bigg)^{1\over2} \sum_{m=0}^{+\infty}\frac{L_{\rho,m+4}  \norm{\partial_x^{m}\lambda}_{L_x^\infty L_y^2}}{m+1}  \leq C |\vec a|_{X_\rho}^3.
 	\end{multline*}
 By substituting the above inequality and \eqref{splfirs} into \eqref{spl} and observing that $|\vec a|_{X_\rho}\leq |\vec a|_{Y_\rho}$, we obtain the   first estimate in Lemma \ref{lem1}.  

A similar argument  as that  above gives that
	\begin{eqnarray*}
 		\sum_{m=0}^{+\infty}L_{\rho,m+2}^2\Big(\partial_x^m \big[-4(\partial_xu)\lambda-3(\partial_x^2u)^2\big],\,  \partial_x^m\lambda\Big)_{L^2}\leq C |\vec a|_{X_\rho}^3,
 	\end{eqnarray*}
 	and 
   \begin{eqnarray*}
 	\begin{aligned}
 &\sum_{m=0}^{+\infty}L_{\rho,m+2}^2\Big(\partial_x^m \big[-3(\partial_xv)\partial_x^2\partial_y u \big],\,  \partial_x^m\lambda\Big)_{L^2} \\
& \leq  C\sum_{m=0}^{+\infty}\sum_{j=0}^{[m/2]}{m\choose j} L_{\rho,m+2}^2\norm{\comi y ^\ell  \partial_x^{j+2}u}_{L_x^\infty L_y^2} \norm{   \partial_x^{m-j+2}\partial_y u}_{L^2}\norm{\partial_x^m\lambda}_{L^2}\\&
\quad +C\sum_{m=0}^{+\infty}\sum_{j=[m/2]+1}^m{m\choose j} L_{\rho,m+2}^2\norm{\comi y^\ell  \partial_x^{j+2}u}_{L^2}\norm{  \partial_x^{m-j+2}\partial_y u}_{L_x^\infty L_y^2}\norm{\partial_x^m\lambda}_{L^2}\\
 &\leq  C |\vec a|_{X_\rho}^2  |\vec a|_{Z_\rho}
 \leq C |\vec a|_{X_\rho}^3+C |\vec a|_{X_\rho}  |\vec a|_{Z_\rho}^2,
 \end{aligned}
 \end{eqnarray*}
where  $v=-\int_0^y \partial_xu(t,x,\tilde y)d\tilde  y$  is used.
Combining the above estimates and observing  that $|\vec a|_{X_\rho}\leq |\vec a|_{Y_\rho}$ gives  the second estimate in Lemma \ref{lem1}. Hence, the proof of the lemma is complete.
 \end{proof}
 
By the definition of $H$ given in  \eqref{llamd++}, 
  the estimates in Lemmas  \ref{lem2} and \ref{lem1} yield that
  \begin{eqnarray*}
  	\sum_{m=0}^{+\infty}L_{\rho,m+2}^2\inner{\partial_x^m H,\,  \partial_x^m\lambda}_{L^2}\leq C|\vec a|_{X_\rho}|\vec a|_{Z_\rho}^2+  C\big( |\vec a|_{X_\rho}+  |\vec a|_{X_\rho}^2\big)|\vec a|_{Y_\rho}^2.
  \end{eqnarray*}

  Therefore, we have, by the equality 
   \eqref{lambeq}, the following corollary:
 
\begin{corollary}[Estimate for $\lambda$]\label{corlam}
	 Under the same assumption as in Theorem \ref{th:ape}, we have  that
	 \begin{multline*}
	 \frac{1}{2}\frac{d}{dt}\sum_{m=0}^{+\infty}L_{\rho,m+2}^2\norm{\partial_x^{m} \lambda}_{L^2}^2+\sum_{m=0}^{+\infty}L_{\rho,m+2}^2\norm{                          \partial_x^{m}\partial_y \lambda}_{L^2}^2+\sum_{m=0}^{+\infty}L_{\rho,m+2}^2\norm{\partial_x^{m} \lambda}_{L^2}^2\\
	 \leq \rho'\sum_{m=0}^{+\infty} \frac{m+3}{\rho}L_{\rho,m+2}^2\norm{\partial_x^{m} \lambda}_{L^2}^2 +C|\vec a|_{X_\rho}|\vec a|_{Z_\rho}^2+ C\big( |\vec a|_{X_\rho}+  |\vec a|_{X_\rho}^2\big)|\vec a|_{Y_\rho}^2.
	\end{multline*}
\end{corollary}

We now turn to derive the estimate on $\mathcal U$. 

\begin{lemma}[Estimate on $\mathcal U$]\label{lem3}
	 Under the same assumption as in Theorem \ref{th:ape}, we have
	\begin{multline*}
	 \frac{1}{2}\frac{d}{dt}\sum_{m=0}^{+\infty}L_{\rho,m+3}^2\norm{\partial_x^{m} \mathcal U}_{L^2}^2+\sum_{m=0}^{+\infty}L_{\rho,m+3}^2\norm{                          \partial_x^{m}\partial_y \mathcal U}_{L^2}^2+\frac12 \sum_{m=0}^{+\infty}L_{\rho,m+3}^2\norm{\partial_x^{m} \mathcal U}_{L^2}^2\\
	 \leq \frac{1}{2}\sum_{m=0}^{+\infty}  L_{\rho,m+2}^2 \norm{\partial_{x}^{m}\lambda}_{L^2}^2+C|\vec a|_{X_\rho} |\vec a|_{Y_\rho}^2.
		\end{multline*}

\end{lemma}

\begin{proof}
By \eqref{reg:mu}	 and the boundary condition that $\partial_y\mathcal U|_{y=0}=\mathcal U|_{y=+\infty}=0,$  the energy  estimate on \eqref{ymau} gives
\begin{multline*}
	 \frac{1}{2}\frac{d}{dt}\sum_{m=0}^{+\infty}L_{\rho,m+3}^2\norm{\partial_x^{m} \mathcal U}_{L^2}^2+\sum_{m=0}^{+\infty}L_{\rho,m+3}^2\norm{                          \partial_x^{m}\partial_y \mathcal U}_{L^2}^2+\sum_{m=0}^{+\infty}L_{\rho,m+3}^2\norm{\partial_x^{m} \mathcal U}_{L^2}^2\\
	 \leq  \rho' \sum_{m=0}^{+\infty} \frac {m+4 }{\rho }L_{\rho,m+3}^2 \norm{\partial_{x}^{m}\mathcal U}_{L^2}^2+\sum_{m=0}^{+\infty}  L_{\rho,m+3}^2 \norm{\partial_{x}^{m+1}\lambda}_{L^2}\norm{\partial_{x}^{m}\mathcal U}_{L^2} \\
		      +\sum_{m=0}^{+\infty}L_{\rho,m+3}^2\Big(\partial_x^m \Big[(\partial_x\partial_yu)\int_0^y\mathcal U  d\tilde y+(\partial_xu)\mathcal U \Big],\    \partial_x^m\mathcal U\Big)_{L^2}.
	\end{multline*}
Note that  the first term on the right side is non-positive and the second one  is bounded from above by
\begin{eqnarray*}
	\frac{1}{2}\sum_{m=0}^{+\infty}  L_{\rho,m+2}^2 \norm{\partial_{x}^{m}\lambda}_{L^2}^2+\frac12\sum_{m=0}^{+\infty}  L_{\rho,m+3}^2  \norm{\partial_{x}^{m}\mathcal U}_{L^2}^2.
\end{eqnarray*}
For the last term in the above inequality,  
 a similar argument as that for Lemmas \ref{lem2} and \ref{lem1}  yields that
\begin{eqnarray*}
	\begin{aligned}
	\sum_{m=0}^{+\infty}L_{\rho,m+3}^2\Big(\partial_x^m \Big[(\partial_x\partial_yu)\int_0^y\mathcal U  d\tilde y+(\partial_xu)\mathcal U \Big],\    \partial_x^m\mathcal U\Big)_{L^2} 	
	\leq  C|\vec a|_{X_\rho}^3\leq C |\vec a|_{X_\rho} |\vec a|_{Y_\rho}^2.
\end{aligned}
\end{eqnarray*}
Combining    the above estimates completes the proof of the lemma.
% together  Lemma \ref{lem3}, completing the proof.  
\end{proof}

\begin{proof}
	[\bf Proof of Proposition \ref{prpaux}] Combining the estimates in Corollary \ref{corlam} and Lemma \ref{lem3} gives that
	 \begin{equation}\label{uplamd}
		\begin{aligned}
	&\frac{1}{2}\frac{d}{dt}\sum_{m=0}^{+\infty}\bigg(L_{\rho,m+2}^2\norm{\partial_x^{m} \lambda}_{L^2}^2+L_{\rho,m+3}^2\norm{\partial_x^{m} \mathcal U}_{L^2}^2\bigg)\\
	&\quad+\sum_{m=0}^{+\infty}\bigg(L_{\rho,m+2}^2\norm{                          \partial_x^{m}\partial_y \lambda}_{L^2}^2+L_{\rho,m+3}^2\norm{\partial_x^{m} \partial_y\mathcal U}_{L^2}^2\bigg) +\frac12\sum_{m=0}^{+\infty}\bigg(L_{\rho,m+2}^2\norm{\partial_x^{m} \lambda}_{L^2}^2+L_{\rho,m+3}^2\norm{\partial_x^{m} \mathcal U}_{L^2}^2\bigg)\\
	&\leq \rho'\sum_{m=0}^{+\infty} \frac{m+3}{\rho}L_{\rho,m+2}^2\norm{\partial_x^{m} \lambda}_{L^2}^2+ C |\vec a|_{X_\rho}|\vec a|_{Z_\rho}^2+  C\big( |\vec a|_{X_\rho}+  |\vec a|_{X_\rho}^2\big)|\vec a|_{Y_\rho}^2\\
	&\leq \frac{\rho'}{2} \sum_{m=0}^{+\infty} \frac{m+3}{\rho}L_{\rho,m+2}^2\norm{\partial_x^{m} \lambda}_{L^2}^2+\frac{\rho'^3}{4} \sum_{m=0}^{+\infty} \frac{m+3}{\rho}L_{\rho,m+2}^2\norm{\partial_x^{m} \lambda}_{L^2}^2  + C\inner{ |\vec a|_{X_\rho}+  |\vec a|_{X_\rho}^2}\inner{|\vec a|_{Y_\rho}^2+|\vec a|_{Z_\rho}^2},
	\end{aligned}
	\end{equation}
	where we have used  the fact that  $\rho'/2\leq\rho'^3/2\leq \rho'^3/4$ in 
	the last inequality.  
	
	It remains to estimate the first term on the right hand side of \eqref{uplamd} as follows:
		\begin{equation}\label{kee}
		\begin{aligned}
&\frac12	\rho'\sum_{m=0}^{+\infty} \frac{m+3}{\rho}L_{\rho,m+2}^2\norm{\partial_x^{m} \lambda}_{L^2}^2+\frac{1}{4} \frac{d}{dt} \sum_{m=0}^{+\infty}  \rho'^2 \frac{(m+4)^2}{\rho^2}L_{\rho,m+3}^2\norm{   \partial_{x}^{m}\mathcal U}_{L^2}^2\\ 
	&\qquad+ \frac12 \rho'^2\sum_{m=0}^{+\infty} \frac{(m+4)^2}{\rho^2}  L_{\rho,m+3} ^2\norm{ \partial_{x}^{m}\mathcal U }_{L^2}^2\\
		&\leq \frac{1}{4} \rho'^3 \sum_{m=0}^{+\infty}\frac{(m+4)^3}{\rho^3}L_{\rho,m+3}^2 \norm{\partial_{x}^{m}\mathcal U}_{L^2}^2 
		      +C|\vec a|_{X_\rho}^2|\vec a|_{Z_\rho}^2+ C |\vec a|_{X_\rho}^2  |\vec a|_{Y_\rho}^2. 
		      \end{aligned}
	\end{equation}
Then,   Proposition \ref{prpaux} follows by combining  \eqref{uplamd} and \eqref{kee}.

To  prove \eqref{kee},  denote that
 \begin{equation}\label{def:p}
 P=\partial_t+u\partial_x +v\partial_y-\partial_y^2+1.
 \end{equation}
By applying  $\partial_x^m$ to \eqref{ymau}, we obtain that
\begin{equation}\label{lamu}
	\partial_x^{m+1}\lambda=P\partial_x^{m} \mathcal U+K_m,
\end{equation}
where
\begin{equation}\label{rekm}
	K_m=\sum_{j=1}^m{m\choose j}\big[(\partial_x^j u)\partial_{x}^{m-j+1} \mathcal U+(\partial_x^j v)\partial_{x}^{m-j}\partial_y \mathcal U \big]-\partial_x^{m}\Big[ (\partial_x\partial_yu)\int_0^y\mathcal U d\tilde y+ (\partial_xu)\mathcal U\Big].
\end{equation}
Then, by 
 observing that  
 \begin{eqnarray*}
	\frac12\rho'\sum_{m=0}^{+\infty} \frac{m+3}{\rho}L_{\rho,m+2}^2\norm{\partial_x^{m} \lambda}_{L^2}^2\leq   \frac12	\rho'\sum_{m=0}^{+\infty} \frac{m+4}{\rho}L_{\rho,m+3}^2\norm{\partial_x^{m+1} \lambda}_{L^2}^2,
\end{eqnarray*}
and 
\begin{eqnarray*}
  \frac12	\rho' \norm{\partial_x^{m+1} \lambda}_{L^2}^2+ \frac12	\rho' \norm{K_m}_{L^2}^2\leq  \frac14	\rho'\norm{P\partial_x^{m} \mathcal U}_{L^2}^2,
\end{eqnarray*}
because of \eqref{lamu}    and  the fact that $\rho'<0,$  
we   have that
\begin{equation}\label{tla}
  \frac12\rho'\sum_{m=0}^{+\infty} \frac{m+3}{\rho}L_{\rho,m+2}^2\norm{\partial_x^{m} \lambda}_{L^2}^2 
	 \leq  \frac14 \rho'  \sum_{m=0}^{+\infty} \frac{m+4}{\rho}L_{\rho,m+3}^2\norm{P\partial_x^{m} \mathcal U }_{L^2}^2-	 \frac12 \rho'\sum_{m=0}^{+\infty} \frac{m+4}{\rho}L_{\rho,m+3}^2\norm{K_m}_{L^2}^2.
\end{equation}

The two terms on the right hand side of the above inequality will be estimated as follows:

\noindent
 {\bf (i) Estimate on the last term on the right hand side of  \eqref{tla}}.   Noting that $ -\rho'/\rho\leq 1$,  by the defintion of $K_m$  in \eqref{rekm}, we have that
 \begin{equation}\label{fines}
 \begin{aligned}
	- \frac12\rho'\sum_{m=0}^{+\infty} \frac{m+4}{\rho}L_{\rho,m+3}^2\norm{K_m}_{L^2}^2\leq &  \sum_{m=0}^{+\infty}(m+4)  L_{\rho,m+3}^2\norm{K_m}_{L^2}^2\leq    S_1+S_2+S_3,
		\end{aligned}
	\end{equation}
where
\begin{eqnarray*}%\label{s1s2}
%\left\{
	\begin{aligned}
S_1&=4 \sum_{m=0}^{+\infty}   \bigg( (m+4)^{1/2} L_{\rho,m+3}\sum_{j=1}^{m}    \frac{m!}{j!(m-j)!}      \norm{   (\partial_{x}^j u ) \partial_{x}^{m-j+1}\mathcal U}_{L^2} \bigg)^2,\\  
	S_2&=4\sum_{m=0}^{+\infty}   \bigg((m+4)^{1/2}L_{\rho,m+3}\sum_{j=1}^{m}    \frac{m!}{j!(m-j)!}      \norm{   (\partial_{x}^j v ) \partial_{x}^{m-j}\partial_y\mathcal U}_{L^2} \bigg)^2,  
	\end{aligned}
%\right.
\end{eqnarray*}
and 
\begin{eqnarray*}
	S_3=4\sum_{m=0}^{+\infty}   \bigg( (m+4)^{1/2}  L_{\rho,m+3}       \big\|  \partial_x^{m}\big[ (\partial_x\partial_yu)\int_0^y\mathcal U d\tilde y+ (\partial_xu)\mathcal U\big]\big\|_{L^2} \bigg)^2.
\end{eqnarray*} 
We first estimate  $S_1$. Note that 
\begin{eqnarray*}
	\forall\ 1\leq j\leq [m/2], \quad      (m+4)^{1/2} \frac{m!}{j!(m-j)!} \frac{L_{\rho,m+3}}{L_{\rho,j+3}L_{\rho,m-j+4}}   \leq   \frac{C}{ j+1}\frac{(m-j+5)^{3/2}}{\rho^{3/2}},
\end{eqnarray*}
and
\begin{eqnarray*}
	\forall\   [m/2]+1\leq j\leq m, \quad      (m+4)^{1/2} \frac{m!}{j!(m-j)!} \frac{L_{\rho,m+3}}{L_{\rho,j+1}L_{\rho,m-j+6}}   \leq  \frac{C}{m-j+1}.
\end{eqnarray*}
By \eqref{dis}, we obtain that
\begin{eqnarray*}
	\begin{aligned}
		S_1&\leq C \sum_{m=0}^{+\infty}   \bigg(  \sum_{j=1}^{[m/2]}          \frac{L_{\rho, j+3}\norm{    \partial_{x}^j u }_{L^\infty}}{j+1} \frac{(m-j+5)^{3/2}}{\rho^{3/2}} L_{\rho,m-j+4}\norm{   \partial_{x}^{m-j+1}\mathcal U}_{L^2} \bigg)^2\\
	&\quad+C	\sum_{m=0}^{+\infty}   \bigg(  \sum_{j= [m/2]+1}^m           L_{\rho, j+1}\norm{    \partial_{x}^j u }_{L_x^2L_y^\infty}\frac{  L_{\rho,m-j+6}\norm{   \partial_{x}^{m-j+1}\mathcal U}_{L_x^\infty L_y^2}}{m-j+1} \bigg)^2\\
	&\leq C|\vec a|_{X_\rho}^2|\vec a|_{Y_\rho}^2+C|\vec a|_{X_\rho}^4\leq C|\vec a|_{X_\rho}^2|\vec a|_{Y_\rho}^2.
	\end{aligned}
\end{eqnarray*}
Similarly, we have that
\begin{multline*}
	S_2   \leq \sum_{m=0}^{+\infty}   \bigg(  \sum_{j=1}^{[m/2]}          \frac{L_{\rho, j+3}\norm{  \comi y^\ell  \partial_{x}^{j+1} u }_{L_x^\infty L_y^2}}{j+1}  L_{\rho,m-j+3}\norm{   \partial_{x}^{m-j}\partial_y \mathcal U}_{L^2} \bigg)^2\\
	 \quad+	\sum_{m=0}^{+\infty}   \bigg(  \sum_{j= [m/2]+1}^m           L_{\rho, j+1}\norm{   \comi y^\ell  \partial_{x}^{j+1} u }_{L^2}\frac{  L_{\rho,m-j+5}\norm{   \partial_{x}^{m-j}\partial_y\mathcal U}_{L_x^\infty L_y^2}}{m-j+1} \bigg)^2 
	 \leq  C|\vec a|_{X_\rho}^2|\vec a|_{Z_\rho}^2,
\end{multline*}
and
\begin{eqnarray*}
	S_3\leq C|\vec a|_{X_\rho}^2|\vec a|_{Y_\rho}^2.
\end{eqnarray*}
Substituting the above estimates into \eqref{fines} yields
that
 \begin{equation}\label{ukm}
	-\frac12 \rho'\sum_{m=1}^{+\infty} \frac{m+4}{\rho}L_{\rho,m+3}^2\norm{K_m}_{L^2}^2\leq  C|\vec a|_{X_\rho}^2|\vec a|_{Z_\rho}^2+ C |\vec a|_{X_\rho}^2  |\vec a|_{Y_\rho}^2.
\end{equation}

\noindent
 {\bf (ii) Estimate on the first term on the right hand side of  \eqref{tla}}.   % Recall $P$ is defined by \eqref{def:p}. 
 Direct calculation gives that
 \begin{multline*}
  	\rho'\frac{m+4}{\rho}L_{\rho,m+3}^2 \norm{P \partial_x^{m} \mathcal U}_{L^2}^2 
		    = \rho'\frac{m+4}{\rho}  \norm{ P\big(L_{\rho,m+3} \partial_{x}^{m}\mathcal U\big)}_{L^2}^2+ \rho'^3 \frac{(m+4)^3}{\rho^3}L_{\rho,m+3}^2 \norm{\partial_{x}^{m}\mathcal U}_{L^2}^2\\
		 \quad   -2 \rho'^2 \frac{(m+4)^2}{\rho^2}  \inner{P \big(L_{\rho,m+3} \partial_{x}^{m}\mathcal U\big),\ L_{\rho,m+3}\partial_{x}^{m}\mathcal U}_{L^2}.
 \end{multline*}
 For the last term in the above inequality, by \eqref{def:p}, one has that
 \begin{eqnarray*}
 	\inner{P \big(L_{\rho,m+3} \partial_{x}^{m}\mathcal U\big),\ L_{\rho,m+3}\partial_{x}^{m}\mathcal U}_{L^2}\geq \frac{1}{2}\frac{d}{dt} \norm{ L_{\rho,m+3} \partial_{x}^{m}\mathcal U }_{L^2}^2+L_{\rho,m+3} ^2\norm{ \partial_{x}^{m}\mathcal U }_{L^2}^2.
 \end{eqnarray*}
Thus, %combining the above estimates and observing $\rho'<0,$ we have
\begin{multline*}
		   \rho'\frac{m+4}{\rho}L_{\rho,m+3}^2 \norm{P \partial_x^{m} \mathcal U}_{L^2}^2 		\\
		    \leq   \rho'^3 \frac{(m+4)^3}{\rho^3}L_{\rho,m+3}^2 \norm{\partial_{x}^{m}\mathcal U}_{L^2}^2-2 \rho'^2 \frac{(m+4)^2}{\rho^2}  L_{\rho,m+3} ^2\norm{ \partial_{x}^{m}\mathcal U }_{L^2}^2   -  \rho'^2 \frac{(m+4)^2}{\rho^2} \frac{d}{dt} \norm{ L_{\rho,m+3} \partial_{x}^{m}\mathcal U }_{L^2}^2.
		 \end{multline*}
	Here,  the last term can be written as   
\begin{eqnarray*}
\begin{aligned}
	& -  \rho'^2 \frac{(m+4)^2}{\rho^2} \frac{d}{dt} \norm{ L_{\rho,m+3} \partial_{x}^{m}\mathcal U  }_{L^2}^2 \\
	&\quad  =-  \frac{d}{dt} \Big( \rho'^2 \frac{(m+4)^2}{\rho^2}L_{\rho,m+3}^2\norm{    \partial_{x}^{m}\mathcal U }_{L^2}^2\Big)+2\rho'\inner{ \rho''-\rho'^2/\rho }\frac{(m+4)^2}{\rho^2} L_{\rho,m+3} ^2\norm{  \partial_{x}^{m}\mathcal U }_{L^2}^2\\
	 &\quad\leq -  \frac{d}{dt} \Big( \rho'^2 \frac{(m+4)^2}{\rho^2}L_{\rho,m+3}^2\norm{    \partial_{x}^{m}\mathcal U}_{L^2}^2\Big), 
	 \end{aligned} 
\end{eqnarray*}
where we have used  \eqref{eqvi} in  the last inequality. Hence, 
\begin{multline*}
		  \rho'\frac{m+4}{\rho}L_{\rho,m+3}^2 \norm{P \partial_x^{m} \mathcal U}_{L^2}^2  
		 \leq   \rho'^3 \frac{(m+4)^3}{\rho^3}L_{\rho,m+3}^2 \norm{\partial_{x}^{m}\mathcal U}_{L^2}^2\\ 
		  -2 \rho'^2 \frac{(m+4)^2}{\rho^2}  L_{\rho,m+3} ^2\norm{ \partial_{x}^{m}\mathcal U }_{L^2}^2  -  \frac{d}{dt} \Big( \rho'^2 \frac{(m+4)^2}{\rho^2}L_{\rho,m+3}^2\norm{   \partial_{x}^{m}\mathcal U}_{L^2}^2\Big).
\end{multline*}
Substituting the above estimate and \eqref{ukm} into \eqref{tla} yields \eqref{kee}. Thus, the proof of proposition  is complete.
\end{proof}

 \subsection{Estimate on the tangential velocity}\label{subsec:u}
 
We now derive the estimate on $\norm{u}_{X_\rho}$.% recalling  $\norm{u}_{X_\rho} $  is given in Definition \ref{defgev}.

\begin{proposition}\label{prop:u}
	Under the same assumption as in Theorem \ref{th:ape}, the  estimate  
	\begin{eqnarray*}
	\begin{aligned}
	& \frac{1}{2}\frac{d}{dt} \sum_{m=0}^{+\infty}L_{\rho,m+k}^2\norm{\tau^{\ell+k}\partial_x^{m}\partial_y^k u}_{L^2}^2+\frac14 \sum_{m=0}^{+\infty}L_{\rho,m+k}^2\norm{  \tau^{\ell+k}                       \partial_x^{m}\partial_y^{k+1}u}_{L^2}^2  +\frac14  \sum_{m=0}^{+\infty}L_{\rho,m+k}^2\norm{\tau^{\ell+k}\partial_x^{m} \partial_y^ku}_{L^2}^2 \\
	&  \leq \frac14 \rho'^3 \sum_{m=0}^{+\infty} \frac{m+k+1}{\rho} L_{\rho,m+k}^2 \norm{\tau^{\ell+k}\partial_x^m\partial_y^ku}_{L^2}^2+C \big( |\vec a|_{X_\rho} +|\vec a|_{X_\rho}^2\big) |\vec a|_{Y_\rho}^2 
	  \end{aligned}
	 \end{eqnarray*}
holds for any $0\leq k\leq 3,$ where $\tau$ is defined as in \eqref{tauweigh}.
\end{proposition}

\begin{proof} We will basically give the estimates in the cases of  when $k=0$ and $k=3$.
	
{\bf (a) The case of when \boldmath $k=0$}.  
We first derive the estimate on  tangential derivatives.  
Applying  $  \partial_x^m$ to the first equation in \eqref{2dprandtl} gives that
  \begin{equation*} 
\big(\partial_t+u\partial_x+v\partial_y -\partial_y^2\big)  \partial_x^m  u + \partial_x^m  u
= 
 -\sum_{j=1}^{m}{{m}\choose j}    (\partial_x^j u) \partial_x^{m-j+1} u -\sum_{j=1}^{m-1}{{m}\choose j}     \big(\partial_x^j v\big)\partial_x^{m-j}\partial_y u-(\partial_x^m v )\partial_yu.
    	\end{equation*}
 Multiplying   by   $\tau^{2\ell} \partial_x^mu$ on both sides of the above equation and then  integrating  over $\mathbb R_+^2$, we have, by    observing $|\partial_y\tau^{2\ell}|\leq N^{-1/2}\ell\tau^{2\ell}$	 and $|\partial_y^2\tau^{2\ell}|\leq N^{-1}(\ell+\ell^2)\tau^{2\ell}$, that 
 \begin{equation*}  \begin{aligned}
 &\frac{1}{2} \frac{d}{dt}   \norm{\tau^{\ell}\partial_x^mu}_{L^2}^2+  \big\|\tau^{\ell}\partial_y \partial_x^mu \big\|_{L^2}^2+  \big\| \tau^{\ell}\partial_x^mu \big\|_{L^2}^2\\
& \leq \frac{1}{2}\frac{\ell}{N^{1/2}} \norm{ v\tau^{\ell} \partial_x^mu}_{L^2}\norm{  \tau^{\ell} \partial_x^mu}_{L^2}+ \frac{1}{2}\frac{\ell+\ell^2}{N} \norm{  \tau^{\ell} \partial_x^mu}_{L^2}^2-\sum_{j=1}^{m}{{m}\choose j}   \Big(  \tau^{\ell}  (\partial_x^j u) \partial_x^{m-j+1} u, \  \tau^{\ell}\partial_x^mu\Big)_{L^2}\\
&\quad-\sum_{j=1}^{m-1}{{m}\choose j}  \Big(   \tau^{\ell}   \big(\partial_x^j v\big)\partial_x^{m-j}\partial_y u, \  \tau^{\ell}\partial_x^mu\Big)_{L^2}-\Big(  \tau^{\ell}(\partial_x^m v )\partial_yu, \  \tau^{\ell}\partial_x^mu\Big)_{L^2}.
 \end{aligned}
 \end{equation*}
 This, with \eqref{rel:Nell}, yields that
\begin{equation}\label{dps}
 	\begin{aligned}
 &\frac{1}{2} \frac{d}{dt}  \sum_{m=0}^{+\infty} L_{\rho,m}^2 \norm{\tau^{\ell}\partial_x^mu}_{L^2}^2+     \sum_{m=0}^{+\infty} L_{\rho,m}^2 \|\tau^{\ell}\partial_y \partial_x^mu \|_{L^2}^2+\frac12  \sum_{m=0}^{+\infty} L_{\rho,m}^2    \| \tau^{\ell}\partial_x^mu \|_{L^2}^2\\
& \leq \sum_{m=0}^{+\infty} \rho' \frac{m+1}{\rho} L_{\rho,m}^2 \norm{\tau^{\ell}\partial_x^mu}_{L^2}^2+C  \sum_{m=0}^{+\infty} L_{\rho,m}^2 \norm{ v\tau^{\ell} \partial_x^mu}_{L^2}\norm{  \tau^{\ell} \partial_x^mu}_{L^2}\\
&\quad+\sum_{m=0}^{+\infty} \sum_{j=1}^{m}{{m}\choose j} L_{\rho,m}^2\norm{\tau^\ell  (\partial_x^j u) \partial_x^{m-j+1} u}_{L^2}\norm{\tau^{\ell}\partial_x^mu}_{L^2}\\
&\quad+\sum_{m=0}^{+\infty} \sum_{j=1}^{m-1}{{m}\choose j} L_{\rho,m}^2\norm{\tau^\ell  (\partial_x^j v) \partial_x^{m-j}\partial_y u}_{L^2}\norm{\tau^{\ell}\partial_x^mu}_{L^2} 
+\sum_{m=0}^{+\infty}   L_{\rho,m}^2\norm{\tau^\ell  (\partial_x^m v) \partial_y u}_{L^2}\norm{\tau^{\ell}\partial_x^mu}_{L^2}.
 \end{aligned}
 \end{equation}
Direct calculation gives that
\begin{equation}\label{fiest}
	 \sum_{m=0}^{+\infty} L_{\rho,m}^2 \norm{ v\tau^{\ell} \partial_x^mu}_{L^2}\norm{  \tau^{\ell} \partial_x^mu}_{L^2}\leq C |\vec a|_{X_\rho}^3. 
\end{equation}
Similarly to the proofs of  Lemmas  \ref{lem2} and \ref{lem1},  we have that
\begin{multline}\label{seest}
	\sum_{m=0}^{+\infty} \sum_{j=1}^{m}{{m}\choose j} L_{\rho,m}^2\norm{\tau^\ell  (\partial_x^j u) \partial_x^{m-j+1} u}_{L^2}\norm{\tau^{\ell}\partial_x^mu}_{L^2}\\
 +\sum_{m=0}^{+\infty} \sum_{j=1}^{m-1}{{m}\choose j} L_{\rho,m}^2\norm{\tau^\ell  (\partial_x^j v) \partial_x^{m-j}\partial_y u}_{L^2}\norm{\tau^{\ell}\partial_x^mu}_{L^2}\leq  C |\vec a|_{X_\rho} |\vec a|_{Y_\rho}^2.
\end{multline}
It remains to estimate the last term in \eqref{dps}. Direct calculation gives that
\begin{eqnarray*}
	\forall\ 0\leq m\leq 1,\quad  L_{\rho,m}^2\norm{\tau^\ell  (\partial_x^m v) \partial_y u}_{L^2}\norm{\tau^{\ell}\partial_x^mu}_{L^2}\leq C|\vec a|_{X_\rho}^3. 
\end{eqnarray*}
For $m\geq 2$, by  recalling $\mathscr L(u,\mathcal U)$ given in \eqref{f}, we use \eqref{ma} to write that
\begin{multline*}
	\partial_x^mv=-\int_0^y\partial_x^{m+1} u(t,x,\tilde y) d\tilde y\\=-\int_0^y \partial_x^{m-2} \lambda (t,x,\tilde y) d\tilde y-\partial_x^{m-2} \underbrace{ \int_0^y   \bigg(\partial_yu(t,x,\tilde y)\int_0^{\tilde y} \mathcal U(t,x,r)dr\bigg) d\tilde y}_{=\mathscr L(u,\mathcal U)}.
\end{multline*}
%As a result we combine the above facts to compute,  in view of  Definition \ref{defabnorm},
Thus, 
\begin{eqnarray*}
	\begin{aligned}
		& \sum_{m=0}^{+\infty}   L_{\rho,m}^2\norm{\tau^\ell  (\partial_x^m v) \partial_y u}_{L^2}\norm{\tau^{\ell}\partial_x^mu}_{L^2}\\
		&\leq   \frac{C}{\rho^{5}}|\vec a|_{X_\rho}^3+  C  \sum_{m=2}^{+\infty}   L_{\rho,m}^2\norm{\tau^{\ell+1} \partial_yu}_{L_x^\infty L_y^2 }\norm{ \partial_x^{m-2}\lambda}_{L^2}\norm{\tau^{\ell}\partial_x^mu}_{L^2} \\
		&\quad+ C  \sum_{m=2}^{+\infty}   L_{\rho,m}^2\norm{\tau^{\ell+1} \partial_yu}_{L_x^\infty L_y^2}\norm{  \partial_x^{m-2}\mathscr L(u,\mathcal U)}_{L_x^2 L_y^\infty}\norm{\tau^{\ell}\partial_x^mu}_{L^2}\\
		&\leq C|\vec a|_{X_\rho}^3+C|\vec a|_{X_\rho}\sum_{m=2}^{+\infty}    \frac{(m+2)^{3/2}}{\rho^{3/2}}L_{\rho,m+1}\norm{ \partial_x^{m-2}\mathscr L(u,\mathcal U)}_{L_x^2 L_y^\infty}\frac{(m+1)^{1/2}}{\rho^{1/2}}L_{\rho,m}\norm{\tau^{\ell}\partial_x^mu}_{L^2}\\
		&\leq C \inner{|\vec a|_{X_\rho}+|\vec a|_{X_\rho}^2}  |\vec a|_{Y_\rho}^2,
	\end{aligned}
\end{eqnarray*}
where we have used  \eqref{prest} in  the last inequality.  By substituting the above estimates and \eqref{fiest}-\eqref{seest} into \eqref{dps},  and  by using the fact that
 $\rho'\leq\rho'^3\leq \rho'^3/4\leq 0$, we have that
\begin{multline*}
  \frac{1}{2} \frac{d}{dt}  \sum_{m=0}^{+\infty} L_{\rho,m}^2 \norm{\tau^{\ell}\partial_x^mu}_{L^2}^2+     \sum_{m=0}^{+\infty} L_{\rho,m}^2  \|\tau^{\ell}\partial_y \partial_x^mu \|_{L^2}^2+\frac12  \sum_{m=0}^{+\infty} L_{\rho,m}^2  \| \tau^{\ell}\partial_x^mu  \|_{L^2}^2\\
  \leq \frac14 \rho'^3\sum_{m=0}^{+\infty} \frac{m+1}{\rho} L_{\rho,m}^2 \norm{\tau^{\ell}\partial_x^mu}_{L^2}^2+C\big( |\vec a|_{X_\rho} +|\vec a|_{X_\rho}^2\big) |\vec a|_{Y_\rho}^2.
 \end{multline*}
Thus,  Proposition \ref{prop:u} holds for $k=0.$

\medskip
\noindent
{\bf (b) The case of \boldmath $k=3$}. 
We apply $ \partial_x^m \partial_y^{2} $ to the first equation in \eqref{2dprandtl} 
to obtain that
\begin{eqnarray}\label{umau4}
	 \partial_t  \partial_x^m \partial_y^{2}u
	-  \partial_x^{m}\partial_y^{4}u+  \partial_x^{m}\partial_y^{2}u= 
	 -\partial_x^m\partial_y^{2}\big(u\partial_{x}u+ v \partial_{y}u\big). 
\end{eqnarray}
The assumption \eqref{apasu} implies that each term in the above equation belongs to $  L^2\Big([0,+\infty[; L^2( {\comi y^{\ell+3}})\Big)$ with
\begin{eqnarray*}
	L^2( {\comi y^{\ell+3}}):=\set{f\in \mathcal S';\  \comi y^{\ell+3}f\in L^2}.
\end{eqnarray*}
 Then we multiply  both sides of \eqref{umau4} by $-\partial_y\Big(\tau^{2 (\ell+3)} \partial_x^m\partial_y^{3}u\Big)\in L^2\Big([0,+\infty[; L^2(\comi y^{-(\ell+3)})\Big)$, and then integrate over $\mathbb R_+^2$. %this gives, using  integration  by parts for normal direction  and    observing 
 By using $\partial_y^{2}u |_{y=0}=0 $, we have that
 \begin{equation*}
\begin{aligned}
& \frac{1}{2}\frac{d}{dt}\norm{\tau^{\ell+3}\partial_x^{m}\partial_y^{3}u}_{L^2}^2+\norm{                  \tau^{\ell+3}\partial_x^{m}\partial_y^{4}u }_{L^2}^2+\norm{\tau^{\ell+3}\partial_x^{m}\partial_y^{3}u}_{L^2}^2 \\
&=  -\int_{\mathbb R_+^2}\big(\partial_x^{m}\partial_y^{4}u\big)(\partial_y\tau^{2\ell+6})\partial_x^{m}\partial_y^{3}u\, dxdy +\int_{\mathbb R_+^2} \Big[\partial_x^m\partial_y^{2}\big(u\partial_xu+v\partial_yu\big) \Big]\partial_y\Big(\tau^{2 (\ell+3)} \partial_x^m\partial_y^{3}u\Big)\, dxdy. \end{aligned}
\end{equation*}
For the terms on the right hand side of the above equality,  we have  that
\begin{eqnarray*}
  -\int_{\mathbb R_+^2}\big(\partial_x^{m}\partial_y^{4}u\big)(\partial_y\tau^{2\ell+6})\partial_x^{m}\partial_y^{3}u\,dxdy
		\leq \frac{1}{2}\inner{\| \tau^{\ell+3}\partial_x^m  \partial_y^4 u\|_{L^2}^2 + \| \tau^{\ell+3}\partial_x^m  \partial_y^3 u \|_{L^2}^2},
		\end{eqnarray*}
where we have used the fact that $|\partial_y\tau^{2\ell+6}|\leq \frac{\ell+3}{N^{1/2}} \tau^{2\ell+6}\leq  \tau^{2\ell+6}$, and
\begin{multline*}
	\int_{\mathbb R_+^2} \Big[\partial_x^m\partial_y^{2}\big(u\partial_xu+v\partial_yu\big) \Big]\partial_y\Big(\tau^{2 (\ell+3)} \partial_x^m\partial_y^{3}u\Big)\,dxdy\\
	\leq \frac{1}{4}\inner{\| \tau^{\ell+3}\partial_x^m  \partial_y^4 u\|_{L^2}^2 + \| \tau^{\ell+3}\partial_x^m  \partial_y^3 u \|_{L^2}^2}+4\norm{\tau^{\ell+3}\partial_x^m\partial_y^{2}\big(u\partial_xu+v\partial_yu\big) }_{L^2}^2.
\end{multline*}
%As a result, combining the above estimates implies
Hence,
\begin{eqnarray*} 
\begin{aligned}
&\frac{1}{2}\frac{d}{dt}\sum_{m=0}^{+\infty}L_{\rho,m+3}^2\norm{\tau^{\ell+3}\partial_x^{m}\partial_y^{3}u}_{L^2}^2+\frac{1}{4}\sum_{m=0}^{+\infty}L_{\rho,m+3}^2\norm{ \tau^{\ell+3}\partial_x^{m}\partial_y^{4}u  }_{L^2}^2 +\frac14\sum_{m=0}^{+\infty}L_{\rho,m+3}^2\norm{\tau^{\ell+3}\partial_x^{m}\partial_y^{3}u}_{L^2}^2\\ 
& \leq  \rho^{\prime}\sum_{m=0}^{+\infty}\frac{m+4}{\rho}L_{\rho,m+3}^2\norm{\tau^{\ell+3}\partial_x^{m}\partial_y^{3}u}_{L^2}^2 
 +4\sum_{m=0}^{+\infty}  L_{\rho,m+3} ^2\norm{\tau^{\ell+3}\partial_x^m\partial_y^{2}\big(u\partial_xu+v\partial_yu\big) }_{L^2}^2.
\end{aligned}
\end{eqnarray*} 
Moreover, similarly to the proof of Lemma \ref{lem2},  we can show that
\begin{eqnarray*}
4 \sum_{m=0}^{+\infty}  L_{\rho,m+3} ^2\norm{\tau^{\ell+3}\partial_x^m\partial_y^{2}\big(u\partial_xu+v\partial_yu\big) }_{L^2}^2 \leq C|\vec a|_{X_\rho}^4\leq C|\vec a|_{X_\rho}^2|\vec a|_{Y_\rho}^2.\end{eqnarray*}
Therefore, by noting that  $\rho'\leq\rho'^3\leq \rho'^3/4,$ it holds that
\begin{eqnarray*}
	\begin{aligned}
		&\frac{1}{2}\frac{d}{dt}\sum_{m=0}^{+\infty}L_{\rho,m+3}^2\norm{\tau^{\ell+3}\partial_x^{m}\partial_y^{3}u}_{L^2}^2+\frac{1}{4}\sum_{m=0}^{+\infty}L_{\rho,m+3}^2\norm{ \tau^{\ell+3}\partial_x^{m}\partial_y^{4}u  }_{L^2}^2 +\frac{1}{4}\sum_{m=0}^{+\infty}L_{\rho,m+3}^2\norm{\tau^{\ell+3}\partial_x^{m}\partial_y^{3}u}_{L^2}^2\\  
		& \leq \frac14 \rho'^{3}\sum_{m=0}^{+\infty}\frac{m+3}{\rho}L_{\rho,m+3}^2\norm{\tau^{\ell+3}\partial_x^{m}\partial_y^{3}u}_{L^2}^2+C|\vec a|_{X_\rho}^2|\vec a|_{Y_\rho}^2.
	\end{aligned}
\end{eqnarray*}
This proves  Proposition \ref{prop:u} for $k=3$.

 The cases of when $k=1, 2$ can discussed similarly,  so we omit the details, for brevity. 
The proof of proposition    is complete. 
\end{proof}
 
\subsection{A priori estimate in 2D} \label{subsec:com}

We will apply the following abstract version of the bootstrap principle given in \cite{MR2233925}
to prove Theorem \ref{thmapri}:

 \begin{proposition}[Proposition 1.21 of  \cite{MR2233925}]\label{probootst}
Letting $I$ be a time interval,  and for   each $t\in I$  we have two  statements: a ``hypothesis" $\boldsymbol{H}(t)$ and  a ``conclusion" $\boldsymbol{C}(t)$. Suppose that we can verify the following four statements:  
\begin{enumerate}[(i)]
\item  If $		\boldsymbol{H}( t)$ is true for some time $t\in I$, then  $\boldsymbol{C}( t) $  is also true for the time $t.$
		\item If   $\boldsymbol{C}(t)$ is true for some $t\in I$, then   $\boldsymbol{H}(t')$ holds for all $t'$ in a neighborhood of  $t$.
		\item If $t_1,t_2, \ldots$ is a sequence of times in $I$ which converges to another time $t\in I$ and     $\boldsymbol{C}(t_n)$ is true for all $ t_n$, then   $\boldsymbol{C}(t)$ is true.
		\item  $\boldsymbol{H}(t)$ is true for at least one time  $t\in I.$
	\end{enumerate}
	Then   $\boldsymbol{C}( t)$ is true for all $t\in I.$
\end{proposition}

 % Now we apply  the above  bootstrap principle  to  prove Theorem \ref{thmapri}.   
 For each $T\in [0,+\infty[$, let $\boldsymbol{H}(T)$	be the statement 
\begin{equation}\label{has}
 \forall \ t\in[0, T],\quad 	e^{t/ 2}|\vec a(t)|_{X_{\rho(t)}}^2+\frac14\int_0^t 	e^{s/ 2}\abs{\vec a(s)}_{Z_{\rho(s)}}^2ds\leq    \frac{2\big (1+\rho_0^2\big )}{ \rho_0^2}\eps_0^2,
 \end{equation}
  and let  $\boldsymbol{C}(T)$	be the statement  
\begin{equation}\label{assc}
	\forall \ t\in[0, T],\quad 	e^{t/ 2}|\vec a(t)|_{X_{\rho(t)}}^2+\frac14\int_0^t 	e^{s/ 2}\abs{\vec a(s)}_{Z_{\rho(s)}}^2ds\leq    \frac{1+\rho_0^2}{\rho_0^2}\eps_0^2,
\end{equation}
where $\rho_0,\eps_0$ are the constants given in Theorem \ref{th:ape}. 
In the   discussion that follows, we will verify that the conditions (i)-(iv) in Proposition \ref{probootst} are satisfied by $\boldsymbol{H}(T)$ and $\boldsymbol{C}(T)$ defined as in \eqref{has} and \eqref{assc}.  
Note that  $\vec a|_{t=0}=(u_0, 0,\partial_x^3u_0)$. Thus,  $\boldsymbol{H}(0)$ holds because of  \eqref{smallness} and the fact that 
\begin{multline}\label{intianorm}
	|\vec a(0)|_{X_{\rho_0}}^2 = \norm{u_0}_{X_{\rho_0}}^2+\sum_{m=0}^{+\infty}L_{\rho_0,m+2}^2\norm{\partial_x^{m+3}u_0 }_{L^2}^2\\\leq \norm{u_0}_{X_{\rho_0}}^2+\sum_{m=0}^{+\infty}\frac{(m+3)^4}{4\rho_0^2}4^{-(m+3)}L_{2\rho_0,m+3}^2\norm{\partial_x^{m+3}u_0 }_{L^2}^2\leq \frac{\rho_0^2+1}{\rho_0^2} \norm{u_0}_{X_{2\rho_0}}^2.
\end{multline}
 Hence, the condition (iv) in Proposition \ref{probootst} holds. Moreover,    the conditions (ii)-(iii)
  follow from the  continuity of 
   the function 
  \begin{eqnarray*}
  	t\mapsto e^{t/ 2}|\vec a(t)|_{X_{\rho(t)}}^2+\frac14\int_0^t 	e^{s/ 2}\abs{\vec a(s)}_{Z_{\rho(s)}}^2ds.
  \end{eqnarray*}
   It remains to check the condition (i) in Proposition \ref{probootst}; that is that   \begin{eqnarray*}
   		\boldsymbol{H}( T) \ \textrm{is true for some time } T>0 \Longrightarrow \boldsymbol{C}( T) \ \textrm{is also true for the same time } T.
   \end{eqnarray*}
     By Definition \ref{defabnorm}, we  combine the estimates given in Propositions \ref{prpaux}  and \ref{prop:u} to get that
     \begin{eqnarray*}
     \begin{aligned}
& \frac12 	\frac{d}{dt}|\vec a|_{X_\rho}^2+\frac14|\vec a|_{Z_\rho}^2+\frac14|\vec a|_{X_\rho}^2 \\
&\qquad+ \frac14 \frac{d}{dt}\sum_{m=0}^{+\infty}    \rho'^2 \frac{(m+4)^2}{\rho^2}L_{\rho,m+3}^2\norm{   \partial_{x}^{m}\mathcal U}_{L^2}^2+\frac{1}{8} \sum_{m=0}^{+\infty} \rho'^2 \frac{(m+4)^2}{\rho^2}  L_{\rho,m+3} ^2\norm{ \partial_{x}^{m}\mathcal U }_{L^2}^2\\
& \leq \frac14\rho'^3|\vec a|_{Y_\rho}^2+ C\inner{ |\vec a|_{X_\rho}+  |\vec a|_{X_\rho}^2} |\vec a|_{Z_\rho}^2 + C\inner{ |\vec a|_{X_\rho}+  |\vec a|_{X_\rho}^2} |\vec a|_{Y_\rho}^2 \\
&\leq C\big( |\vec a|_{X_\rho}+  |\vec a|_{X_\rho}^2\big)|\vec a|_{Z_\rho}^2+  \Big(\frac14\rho'^3+C|\vec a|_{X_\rho} +C|\vec a|_{X_\rho}^2\Big)|\vec a|_{Y_\rho}^2.
  	\end{aligned}
  \end{eqnarray*}   
%For the terms in the last line, it follows from the validity of \eqref{has} and the definition \eqref{derho} of $\rho$  as well as \eqref{eqvi} that, for any $t\in [0,T],$
Note that 
\begin{eqnarray*}
	C\big( |\vec a|_{X_\rho}+  |\vec a|_{X_\rho}^2\big)|\vec a|_{Z_\rho}^2\leq 4C\frac{1+\rho_0^2}{\rho_0^{2}} \eps_0|\vec a|_{Z_\rho}^2 \leq  \frac{1}{8}|\vec a|_{Z_\rho}^2 
\end{eqnarray*}
and that
\begin{eqnarray*}
	\frac14\rho'^3+C|\vec a|_{X_\rho} +C |\vec a|_{X_\rho}^2\leq -\frac{1}{4}\bigg(\frac{\rho_0}{24}\bigg)^3e^{-t/4}+ 4C\frac{1+\rho_0^2}{\rho_0^{2}}e^{-t/4}\eps_0\leq 0,
\end{eqnarray*}
provided  $\eps_0$ is sufficiently small.   Combining the above estimates implies that
\begin{eqnarray*}
	 \frac12 	\frac{d}{dt} e^{t/2}|\vec a|_{X_\rho}^2+\frac18 e^{t/2}|\vec a|_{Z_\rho}^2 + \frac14 \frac{d}{dt} e^{t/2} \sum_{m=0}^{+\infty}    \rho'^2 \frac{(m+4)^2}{\rho^2}L_{\rho,m+3}^2\norm{   \partial_{x}^{m}\mathcal U}_{L^2}^2\leq 0.
\end{eqnarray*}
 By integrating the above estimate over $[0,t]$ for any $t\in[0,T]$ and using $\mathcal U|_{t=0}=0$, we obtain that
  \begin{eqnarray*} 
 \forall\ t\in [0,T],\quad 
  e^{t/2}  |\vec a(t)|_{X_{\rho(t)}}^2+\frac14\int_0^t  e^{s/2}  \abs{\vec a(s)}_{Z_{\rho(s)}}^2 ds \leq  |\vec a(0)|_{X_{\rho_0}}^2
 \leq \frac{1+\rho_0^2}{\rho_0^2}\eps_0^2,
    \end{eqnarray*}
    where we have used \eqref{intianorm}.  
  This yields \eqref{assc} if \eqref{has} holds, so that   the condition (i) holds. Therefore, by Proposition \ref{probootst},  the estimate \eqref{assc} holds  for any $T\geq 0,$   and the proof of Theorem \ref{thmapri} is complete.

\section{A priori estimate in 3D}
 The discussion on the 3D magnetic Prandtl model is  similar to that of the 2D case. 
 %The main difference 
 %comes from the auxiliary functions.
 For this,  we will use vector-valued auxiliary functions instead of the scalar ones used in the previous section.  More precisely,  
denote by $\boldsymbol{u}=(u_{1},u_{2})$ and $v$  the tangential and normal velocities, respectively,   and   by $(x,y) $ the spatial variables in $\mathbb R^2\times \mathbb R_+ $ with $x=(x_{1},x_{2})$.    As the counterparts   of $ \mathcal U$ and $\lambda$ defined by \eqref{mau} and  \eqref{ma},  we define   
    $ \boldsymbol{ \mathcal U}=(\mathcal U_1,\mathcal U_2)$ and     $\boldsymbol{  \lambda } =(\lambda_1,\lambda_2,\lambda_3,\lambda_4) $  as follows: let $\mathcal U_j, j=1,2,$ solve the initial-boundary problem  
 \begin{eqnarray*}%\label{mau++}
 \left\{
 \begin{aligned}
 & \big(\partial_t+\boldsymbol{u}\cdot \partial_x  +v\partial_y-\partial_y^2\big)    \int_0^y\mathcal U_j(t,x,\tilde y)  d\tilde y+\int_0^y\mathcal U_j (t,x,\tilde y) d\tilde y  =  -\partial_{x_j}^3 v,\\
 & \mathcal U_j|_{t=0}=0, \quad \partial_y\mathcal U_j|_{y=0}=\mathcal U_j|_{y\rightarrow+\infty}=0.
 \end{aligned}
 \right.
 \end{eqnarray*} 
 Accordingly,
 set
 \begin{eqnarray*}
 	\left\{
 	\begin{aligned}
 		\lambda_1&=\partial_{x_1}^3u_1- (\partial_yu_1)\int_0^y\mathcal U_1 (t,x,\tilde y) d\tilde y,\quad 
 		\lambda_2=\partial_{x_2}^3u_1- (\partial_yu_1)\int_0^y\mathcal U_2 (t,x,\tilde y) d\tilde y,\\
 		\lambda_3&=\partial_{x_1}^3u_2-  (\partial_yu_2)\int_0^y\mathcal U_1 (t,x,\tilde y) d\tilde y,\quad  \lambda_4=\partial_{x_2}^3u_2-  (\partial_yu_2)\int_0^y\mathcal U_2(t,x,\tilde y)  d\tilde y.
 	\end{aligned}
 	\right.
 \end{eqnarray*}	
 Denote  that
 $$ \vec{ \boldsymbol{ a}} =( \boldsymbol{u}, \boldsymbol{  \mathcal U}, \boldsymbol{\lambda}),$$
 and  define $|\vec{ \boldsymbol{ a}}|_{X_\rho},|\vec{ \boldsymbol{ a}}|_{Y_\rho} $ and $|\vec{ \boldsymbol{ a}}|_{Z_\rho}$ as in Definition \ref{defabnorm}.
 
 Then the{\it  a priori} estimate  in Theorem \ref{thmapri}  also holds for the function $ \vec{ \boldsymbol{ a}} $ as stated in the following theorem:
 
 \begin{theorem}
 	\label{th:tape}
 	Suppose that the initial datum $ \boldsymbol{u}_0  $ in \eqref{dpe}  belongs to  $X_{2\rho_0}$ for some $ \rho_0>0$. Let $ \boldsymbol{u}\in L^\infty \inner{[0,+\infty[;\  X_{ \rho}}$ be a solution to \eqref{dpe}  satisfying that
 	\begin{equation*}
 	\int_0^\infty \big(\norm{ \boldsymbol{ u}(t)}_{X_{\rho(t)}}^2+\norm{\partial_y  \boldsymbol{u}(t)}_{X_{\rho(t)}}^2\big) dt<+\infty, 
 	\end{equation*}
 	where  $\rho$ is defined by  \eqref{derho}.  If  $ \norm{ \boldsymbol{u}_0}_{X_{2\rho_0}}\leq \varepsilon_0 $ for  some    sufficiently small $\eps_0>0$,
 	then   we have that%, %with the notations above, 
 	\begin{eqnarray*}
 		\sup_{t\geq 0} e^{t/4}|\vec{ \boldsymbol{ a}}(t)|_{X_{\rho(t)}}+\bigg(\int_0^{+\infty}e^{t/2}|\vec{ \boldsymbol{ a}}(t)|_{Z_{\rho(t)}}^2dt\bigg)^{1/2}\leq \frac{4(\rho_0+1)}{\rho_0}  \eps_0.
 	\end{eqnarray*}
 	  Thus, 
 	\begin{eqnarray*}
 		\sup_{t\geq 0} e^{t/4}| \boldsymbol{ u}(t)|_{X_{\rho(t)}}+\bigg(\int_0^{+\infty}e^{t/2}| \partial_y\boldsymbol{ u}(t)|_{X_{\rho(t)}}^2dt\bigg)^{1/2}\leq \frac{4(\rho_0+1)}{\rho_0}  \eps_0.
 	\end{eqnarray*}
 \end{theorem}

  	The proof  of Theorem \ref{th:tape} is the same as that of the  2D case,  so, for brevity,   we omit the details .

 \bigskip

{\bf Acknowledgements.}
W.-X. Li's research was supported by NSF of China (Nos. 11871054, 11961160716, 12131017)  and  the  Natural Science Foundation of Hubei Province (No. 
2019CFA007).  T. Yang's research was supported by  the General Research Fund of Hong Kong CityU No. 11304419.

%\bibliographystyle{abbrv}
%\bibliography{Bib-fluid}

\end{document}